\documentclass[a4paper]{article}
\usepackage{a4wide}
\usepackage[toc,page]{appendix}

\usepackage {tikz}
\usetikzlibrary {positioning}
\usepackage[english]{babel}
\usepackage[utf8]{inputenc}
\usepackage{ifthen}
\usepackage{color}
\usepackage{amsmath,amsthm,amssymb,amsfonts}
\usepackage{bbold}
\usepackage{ mathrsfs }
\usepackage{bbm}
\usepackage{graphicx}
\usepackage{psfrag}
\usepackage[normalem]{ulem}
\usepackage{enumitem}   
\usepackage{relsize}
\usepackage{subfig}
\usepackage{mathtools}

\usepackage[unicode,colorlinks=true,pagebackref=false]{hyperref}

\definecolor{change}{rgb}{0,.55,.55}

\newtheorem{Def}{Definition}
\newtheorem{Thm}[Def]{Theorem}
\newtheorem{Rmk}[Def]{Remark}
\newtheorem{Lemma}[Def]{Lemma}

\newtheorem{Ass}[Def]{Assumption}
\newtheorem{Prop}[Def]{Proposition}

\newcommand{\R}{\mathbb{R}}

\newcommand{\N}{\mathbb{N}}

\newcommand\quadV[1]{\left[#1\right]}

\newcommand{\Cspace}{C}
\newcommand{\Lspace}{L}
\newcommand\LinSpace[2]{\mathrm{L}\left(#1;#2\right)}
\newcommand\HilbSchm[2]{\mathrm{L}_{2}\left(#1;#2\right)}
\newcommand{\Wspace}{\mathcal{W}}

\newcommand\indicator{\mathbb{1}}
\newcommand{\Kalman}{\mathfrak{K}}

\newcommand\restr[2]{\ensuremath{\left.#1\right|_{#2}}}
\newcommand\norm[2]{\left\|#1\right\|_{#2}}

\newcommand{\VHilbert}{\mathscr{V}}
\newcommand{\HHilbert}{\mathscr{H}}
\newcommand{\UHilbert}{\mathscr{U}}

\newcommand{\Adrift}{\mathcal{A}}
\newcommand{\Bvol}{\mathcal{B}}
\newcommand{\Qcor}{\mathcal{Q}}
\newcommand{\QeigVectors}{e}
\newcommand{\QeigValues}{q}

\newcommand{\observ}{H}

\newcommand\Frechet[2]{\mathrm{D}^{1}_{#1}#2}
\newcommand\Hesse[2]{\mathrm{D}^{2}_{#1}#2}

\newcommand{\oneSidedLip}{\lambda}
\newcommand{\coerA}{\alpha_{V}}
\newcommand{\coerB}{\alpha_{H}}
\newcommand{\coerC}{\alpha_{0}}
\newcommand{\boundB}{\beta}

\title{On the mean field theory of Ensemble Kalman filters for SPDEs}
\author{Sebastian W. Ertel, TU Berlin}

\begin{document}

	\maketitle
	
	\begin{abstract}
		This paper is concerned with the mathematical analysis of continuous time Ensemble Kalman Filters (EnKBFs) and their mean field limit in an infinite dimensional setting. The signal is determined by a nonlinear Stochastic Partial Differential Equation (SPDE), which is posed in the standard variational setting. Assuming global one-sided Lipschitz conditions and finite dimensional observations we first prove the well posedness of both the EnKBF and its corresponding mean field version. We then investigate the quantitative convergence  of the EnKBF towards its mean field limit, recovering the rates suggested by the law of large numbers for bounded observation functions. The main tool hereby are exponential moment estimates for the empirical covariance of the EnKBF, which may be of independent interest. 
		In the appendix of the paper we investigate the connection of the mean field EnKBF and the Stochastic Filtering Problem. In particular we derive the Feedback Particle Filter for infinite dimensional signals and show that the mean field EnKBF can viewed as its constant gain approximation.
	\end{abstract}

	\tableofcontents
	
	\section{Introduction}
	
	In recent years data driven methods have become increasingly relevant in scientific computing. In particular data assimilation, that is the (optimal) integration of real world data into mathematical models, has become a popular research topic for practitioners and mathematicians alike. As it shares similar objectives to stochastic filtering, which is essentially the discipline of Bayesian estimation of dynamic processes from noisy, potentially incomplete data, many algorithms originating in filtering are used for data assimilation tasks. Vice versa, algorithms for data assimilation of dynamical processes can be viewed through the lens of filtering, the mathematical model is then referred to as the signal and the available data as the observations.\newline
		
	One such algorithm is the Ensemble Kalman Filter (EnKF), which was introduced by Geir Evensen in 94 \cite{Evensen} and employs an ensemble of  interacting particles to estimate the state of a dynamical system.  Since its inception many different variants of the EnKF have been introduced. For an overview and historical context we refer to \cite{ReichStuart}, \cite{Evensen Overview} or \cite{VanLeeuwen20}. The EnKF has by now become one of the most widely used techniques for data assimilation in high dimensional settings, particularly popular amongst practitioners in the geosciences and numerical weather forecasting. Besides its usage for state estimation in dynamical systems, the EnKF and related algorithms have also been applied to parameter estimation in inverse problems \cite{OliverEtAl},\cite{SchillingsStuart}.\newline
	
	While the original EnKF is a discrete time recursion, continuous time counterparts referred to as Ensemble Kalman--Bucy Filters (EnKBFs), were first formulated in various works by Bergemann and Reich \cite{BergemannReich}, and have by now been firmly established in the literature. In many cases they can also be shown to be the limit of their discrete time counterparts for vanishing step size, see e.g. \cite{Lange} and the references found therein. In this paper we will only consider the basic continuous time framework in an infinite dimensional setting. The signal $u$ shall be determined by a Stochastic Partial Differential Equation (SPDE)
	\begin{align}\tag{S}\label{signal}
		\mathrm{d}u_t=\Adrift(u_t)\mathrm{d}t+\Bvol(u_t)\mathrm{d}W_t
		\quad\text{with initial condition}~u_0,
	\end{align}
	posed in the standard variational framework found in e.g. \cite{Pardoux - French},\cite{RoecknerPrevot}. Hereby the noise $W$ is assumed to be some (infinite dimensional) Wiener process and the drift term $\Adrift$ is a (differential) operator satisfying some global one-sided Lipschitz property, such as e.g.
	\begin{itemize}
		\item the Laplacian $\Adrift(v):=\Delta v$, in this case \eqref{signal} is a stochastic heat equation.
		
		\item the $p$-Laplacian $\Adrift(v)
		:=\Delta \left(v |v|^{p-1}\right)$ for any $p>1$, in this case \eqref{signal} is a stochastic porous medium equation.
		
		\item a semilinear reaction-advection-diffusion operator of the form $\Adrift(v):=\mathrm{div}\left( \mathfrak{a}\nabla v\right) + \mathrm{div}\left(\mathfrak{b}~v\right) + f(v) + g(v)$, where $f$ is globally Lipschitz continuous and $g$ is a monotone decreasing function. In particular stochastic reaction-diffusion equations with double-well potentials such as Allen--Cahn equations are covered by our theory. 
	\end{itemize}

	 The observation data $Y$ shall be a continuous time process and is assumed to depend on the signal in the following way
	\begin{align}\tag{O}\label{observation}
		\mathrm{d}Y_t=\observ(u_t)\mathrm{d}t + \Gamma_t \mathrm{d} V_t
		,\quad Y_0=0
		.
	\end{align}

	We only consider finite dimensional observations taking values in $\R^{d_y}$ for some $d_y\in\N$. This is the more practically relevant case and also avoids discussions of the regularity/degeneracy of the observation noise $V$, which is assumed to be white, i.e. some finite dimensional standard Brownian motion. A more thorough discussion of the setting we consider, and the assumptions we have to make, is found in section \ref{section problem setting}.\newline
	
	The EnKBF we consider in this paper is the following system of interacting SPDEs
	\begin{align}\label{EnKBF - introduction}
		\begin{split}
			\mathrm{d}u^i_t
			&=\Adrift(u^i_t)\mathrm{d}t+\Bvol(u^i_t)\mathrm{d}W^{i}_t
			\\&\phantom{=}+
			\frac{1}{N} \sum_{j=1}^{N} u^{j}_t \left(\observ(u^{j}_t)-\frac{1}{N}\sum_{k=1}^{N}\observ(u^{k}_t)\right)' R^{-1}_t\left(\mathrm{d}Y_t-\frac{\observ(u^i_t)+\frac{1}{N}\sum_{k=1}^{N}\observ(u^{k}_t)}{2}
			\mathrm{d}t
			\right)
		\end{split}
	\end{align}
	for $i=1,\cdots,N$, where $(W^i)_{i=1,\cdots,N}$ are independent copies of the Wiener process $\bar{W}$. The EnKBF combines the signal SPDE \eqref{signal} with measurement data $Y$ in order to improve the predictive capabilities of the model. This is done by projecting the discrepancy $\mathrm{d}Y_t-\frac{\observ(u^i_t)+\frac{1}{N}\sum_{k=1}^{N}\observ(u^{k}_t)}{2}
	\mathrm{d}t$ between the actual measured observation increment and the predicted increment onto the interacting ensemble $\left(u^{i}\right)_{i=1,\cdots,N}$. This creates a reaction term that is added onto the signal SPDE and nudges it to better fit the observation data.\newline
	
	 \begin{Rmk}
	 	The system \eqref{EnKBF - introduction} is often referred to as the deterministic EnKBF \cite{On the mathematical theory of linear EnKBF}, which is the continuous time counterpart of the Ensemble Square Root Filter \cite{StannatLange}/deterministic EnKF by Sakov and Oke \cite{SakovOke}. Our main results can be generalized to other types of EnKBFs, in particular the more classical version which involves randomness in the innovation term.
	 \end{Rmk}
 
 	In this paper we investigate the mean field theory of the EnKBF \eqref{EnKBF - introduction} in the variational SPDE setting. The mean field perspective for Ensemble Kalman Filters has become a popular research topic in recent years as it provides a simplified framework for the derivation and mathematical analysis of such algorithms \cite{ReichStuart}. In particular the mean field limit of the EnKBF,
 	\begin{align}\label{mean field EnKBF - introduction}
 		\begin{split}
 			\mathrm{d}u_t
 			&=\Adrift(u_t)\mathrm{d}t+\Bvol(u_t)\mathrm{d}W_t
 			\\&\phantom{=}+
 			\mathbb{Cov}\left[u_t,\observ(u_t)~\bigg|~Y_s,s\leq t\right] R^{-1}_t\left(\mathrm{d}Y_t-\frac{\observ(u^i_t)+
 			\mathbb{E}\left[\observ(u_t)~\big|~Y_s,s\leq t\right]
 			}{2}
 			\mathrm{d}t
 			\right),
 		\end{split}
 	\end{align}
 	 which is  McKean--Vlasov equation that is henceforth just referred to as the mean field EnKBF, gives a connection to solutions of the filtering problem through what is known as the constant gain approximation \cite{TaghvaeiDeWiljesEtAl} of the Feedback Particle Filter (FPF). First quantitative mean field limits of Ensemble Kalman Filters were obtained in the time discrete setting by \cite{LeGlandMonbetTran} and \cite{MandelCobbBeezley}. For the continuous time equations a first quantitative propagation of chaos result was proven in the seminal work \cite{DelMoralTugaut} in a finite dimensional linear Gaussian setting, i.e. when both the signal \eqref{signal} is an Ornstein--Uhlenbeck process and the observation function $\observ$ is assumed to be linear. In this linear Gaussian setting the mean field limit plays a special role as it coincides with the optimal filter described by the classical Kalman--Bucy equations.
 	Several works have expanded on the results of \cite{DelMoralTugaut}, proving time uniform propagation of chaos even in cases of unstable signal dynamics. For an overview of the theory of the EnKBF in the linear Gaussian setting we refer the reader to \cite{On the mathematical theory of linear EnKBF}.\newline
	
	In nonlinear (continuous time) settings only a few convergence results for finite dimensional signals seem to have been obtained so far. The paper \cite{StannatLange} proved a propagation of chaos result with a combination of synchronous coupling and stopping arguments. However they were not able to quantify the convergence rates and had to assume the existence and uniqueness of the mean field limit. The well posedness of of the McKean--Vlasov equation describing the mean field limit of the EnKBF was shown in \cite{Coghi et al} and \cite{ErtelStannat}. Both papers also showed propagation of chaos for the EnKBF. \cite{ErtelStannat} proved a non quantitative convergence result, similar to the aforementioned \cite{StannatLange}. On the other hand \cite{Coghi et al} were even able to prove a mean field limit with logarithmic rates. However they required bounded signal and observation functions, and assumed that the observation data was smoothed (Lipschitz continuous in time). Even though logarithmic convergence is far slower than the desired rates corresponding to the law of large numbers, the result of \cite{Coghi et al} seem to be the best estimate existing so far for nonlinear signals. For related sampling algorithms such as the Ensemble Kalman Sampler, Ensemble Kalman Inversion and mean field Langevin Dynamics various works \cite{BhandariPidstrigachReich}\cite{DingLi1}\cite{DingLi2}\cite{Vaes}\footnote{The preprints \cite{BhandariPidstrigachReich} and \cite{Vaes} actually were uploaded only shortly after the first version of this paper.} have shown quantitative propagation of chaos. While such systems are related to the Ensemble Kalman methodology, the structural differences to the EnKBF seems to not let us simply carry over the employed techniques to prove convergence.\newline
	
	In the infinite dimensional continuous time setting the mean field theory seems to have not been covered by the existing literature, the mathematical analysis, which was initiated in \cite{LawKellyStuart}, so far seems to have focussed only on the interacting particle system \eqref{EnKBF - introduction}. The aim of this paper therefore is to bridge this gap and provide a rather complete mean field theory for EnKBFs in the case of SPDE signals. Throughout this work we shall assume that both the signal coefficients and the observation function satisfy some global Lipschitz conditions described in detail in section \ref{section problem setting}. In this setting our main contributions are the following:
	\begin{itemize}
		\item In section \ref{section Analysis of particle system} we first show the well posedness of the EnKBF \eqref{EnKBF - introduction}. While \cite[Theorem 6.2]{LawKellyStuart} considered well posedness of EnKBFs, even without global Lipschitz conditions, they assumed the existence and uniqueness of solutions and showed a priori estimates, preventing blowups. Furthermore the EnKBF \eqref{EnKBF - introduction} falls outside the standard variational SPDE theory, as found e.g. in \cite{LiuRoeckner-localLipschitzSPDEs}, due to missing growth conditions. We prove well posedness in Theorem \ref{Lemma well posedness particle system}. The main tool hereby are a priori estimates of the empirical variance akin to the law of total variance of conditional expectations. By also controlling the quadratic variation of the empirical variance we are able to prove some exponential moment bounds in Proposition \eqref{Lemma exponential moments}. These bounds are crucial later on for quantitative propagation of chaos, but may also be of independent interest.
		
		\item Next, in section \ref{section mf EnKBF}, we prove the well posedness of the McKean--Vlasov equation describing the mean field limit of the EnKBF in Theorem \ref{well posedness mf EnKBF}. Due to possible non-equivalence of norms in infinite dimensional spaces, we can not simply adapt the existing proofs of the finite dimensional setting, such as e.g. found in \cite{ErtelStannat}, and instead make a fixed point argument with respect to the observation function $\observ$.
		
		\item Finally, in section \ref{section prop of chaos}, we the convergence of the EnKBF towards the mean field limit. First, in Theorem \ref{prop of chaos - implicit rates}, we extend the stopping arguments of \cite{StannatLange} to our infinite dimensional setting to prove convergence with implicit rates. The main result of this paper is found in Theorem \ref{Thm optimal convergence}, where we use the exponential moment bounds of subsection \ref{subsection exponential moment bounds} to prove convergence rates in accordance to the law of large numbers under the additional assumption of bounded observation functions. Even for finite dimensional signals this significantly improves on the at best logarithmic rates shown so far in the literature and actually does so under far less restricting assumptions regarding the signal coefficients and the observation process.
	\end{itemize}
	
	Furthermore we include an appendix which details the connections of the (mean field) EnKBF to the classical stochastic filtering problem. Therefore in section \ref{section Kushner Stratonovich} we briefly recall the Kushner--Stratonovich equation, describing the posterior distribution $(\eta)_{t\geq 0}$, which for continuous time stochastic filtering problems is given by
	\begin{align}\tag{P}\label{posterior}
		\eta_t:=\mathbb{P}\left(~u_t\in\cdot~|~Y_{s},~s\leq t ~\right),~\text{for}~t\geq 0.
	\end{align}
	We also briefly recall the law of total variance in Bayesian statistics, which bounds the posterior variance by variance of the prior. Interestingly, even though generally the (time-)marginal distributions of the (mean field) EnKBF do not coincide with the posterior, except for the linear Gaussian case, it still satisfies similar variance bounds, which actually are the main tool we use for its analysis. Next, in section \ref{section FPF}, we derive the Feedback Particle Filter (FPF)\cite{YangMehtaMeyn1}\cite{YangMehtaMeyn2}\cite{CrisanXiong}, a McKean--Vlasov representation of the Kushner--Stratonovich equation and thus also the posterior, in our infinite dimensional setting. The method hereby is similar to what was done in \cite{PathirajaStannatReich} for finite dimensional signals dimensions. We then show that the mean field EnKBF is the constant gain approximation of the FPF, which was first noted in finite dimensions by \cite{TaghvaeiDeWiljesEtAl}. This connection to the stochastic filtering problem is a strong motivation for the mean field perspective on Ensemble Kalman Filters, as it provides a path to compare them with the optimal Bayesian estimator. Indeed only recently such investigations were started by \cite{CarrilloHoffmannStuartVaes}, estimating the statistical accuracy of discrete time mean field Ensemble Kalman Filters.  Most of the results/methods of the appendix, in particular section \ref{section FPF}, are just simple generalisations of established results in finite dimensional settings. Nevertheless it provides a clear motivation for the importance of the mean field theory of EnKBFs discussed in the main part and shows how such equations can be derived from the Stochastic Filtering Problem.

	\section{Problem setting, assumptions and notations}\label{section problem setting}

	For the signal \eqref{signal} we consider SPDEs in a variational setting as they are found in e.g. \cite{Pardoux},\cite{RoecknerPrevot}. To fix notation and for the convenience of the reader we repeat some key concepts and results of this field in this section. For a more detailed introduction to this topic we refer the reader to \cite{RoecknerPrevot}.\newline
	
	Let $\HHilbert$ be a Hilbert space and $\VHilbert$ be a Banach space that form a Gelfand triple \cite[Section 4.1]{RoecknerPrevot}  $\VHilbert\hookrightarrow\HHilbert \hookrightarrow \VHilbert'$. We denote by $\prescript{}{\VHilbert'}{\left\langle \cdot,\cdot\right\rangle_{\VHilbert}}$ the natural pairing of the Banach space $\VHilbert$ and its dual space $\VHilbert'$. Similarly $\left\langle \cdot,\cdot\right\rangle_{\HHilbert}$ denotes the inner product of the Hilbert space $\HHilbert$. As usual the corresponding norms on $\VHilbert,\HHilbert$ and $\VHilbert'$ are denoted by $\norm{\cdot}{\VHilbert},\norm{\cdot}{\HHilbert}$ and $\norm{\cdot}{\VHilbert'}$. The absolute value as well as the standard Euclidian norm on finite dimensional spaces are both denoted denoted by $|\cdot|$. Furthermore we assume that that there exists an orthonormal basis of $\HHilbert$, denoted by $\left(\nu_k\right)_{k\in\N}\subset\HHilbert$ such that $\nu_k\in\VHilbert$ for all $k\in\N$.\newline
	
	Let $\UHilbert$  be some given separable real Hilbert space and let $\left(\mathfrak{F}_t\right)_{t\geq 0}$ be some given filtration on a probability space $\left(\Omega,\mathfrak{F}_{\infty},\mathbb{P}\right)$. We consider the $\UHilbert$-valued and $\mathfrak{F}$-adapted $\Qcor$-Wiener process  $(W_t)_{t\geq 0 }$ with finite trace. To this end assume that $\Qcor$ is a self-adjoint, positive semidefinite linear operator on $\UHilbert$, with finite trace $\mathrm{tr}_{\UHilbert}\Qcor<+\infty$ and Eigenvalues $\left(\QeigValues_k\right)_{k\in\N}$ with corresponding orthonormal Eigenbasis $\left(\QeigVectors_k\right)_{k\in\N}$. Then there are exist independent $\mathfrak{F}-$adapted Brownian motions $(w^k)_{k\in\N}$, such that
	\begin{align*}
		W_t=\sum_{k\in\N}\sqrt{\QeigValues_k}\QeigVectors_k w^k_t~\text{for all times}~t\geq 0.
	\end{align*}

	\begin{Def}
		We will always identify the Hilbert spaces $\HHilbert,\UHilbert$ with their duals $\HHilbert',\UHilbert'$. For any operator $B$ acting on Hilbert spaces we denote its adjoint by $B'$. The adjoint of an element $u\in\HHilbert$ is just it's image under the Riesz embedding, i.e. $u':=\left\langle u, \cdot\right\rangle_{\HHilbert}$. This notation is consistent with finite dimensional settings. We note that therefore $u u'\in\LinSpace{\HHilbert}{\HHilbert}$ defines a bounded linear operator on $\HHilbert$.
	\end{Def}
	
	To rigorously formulate the signal \eqref{signal} as a variational SPDE on the Gelfand triple $(\VHilbert,\HHilbert,\VHilbert')$, we make the following standard assumptions \cite[page 56]{RoecknerPrevot} that shall hold throughout this paper.

	\begin{Ass}[Signal assumptions]\label{signal assumptions}
		Denote by $\HilbSchm{\UHilbert}{\HHilbert}$ the space of Hilbert--Schmidt operators. That is the space of all linear operators $B:\UHilbert\to\HHilbert$, such that their Hilbert--Schmidt norm $\norm{B}{\HilbSchm{\UHilbert}{\HHilbert}}^2:=\sum_{k\in\N}\norm{B\QeigVectors_k}{\HHilbert}^2$ is finite.\\ 
		We assume that $\Adrift:\VHilbert\to \VHilbert'$ and $\Bvol:\VHilbert\to \HilbSchm{\UHilbert}{\HHilbert}$ satisfy the following conditions:
		\begin{enumerate}
			\item \textbf{Hemicontinuity:} For all $u,v,w\in \VHilbert$ the mapping
			\begin{align*}
				r\to\prescript{}{\VHilbert'}{\left\langle \Adrift(v+r u),w\right\rangle_{\VHilbert}}
				~
				\text{is continuous.}
			\end{align*}
			
			\item \textbf{Weak monotonicity/one-sided Lipschitz:} There exists $\oneSidedLip>0$ such that for all $u,v\in \VHilbert$
			\begin{align}\label{one sided Lipschitz}
				2\prescript{}{\VHilbert'}{\left\langle \Adrift(u)-\Adrift(v),u-v\right\rangle_{\VHilbert}}
				+
				\norm{\left(\Bvol(u)-\Bvol(v)\right)\sqrt{\Qcor}}{ \HilbSchm{\UHilbert}{\HHilbert} }^2
				\leq
				\oneSidedLip \norm{u-v}{\HHilbert}^2.
			\end{align}
			
			\item \textbf{Coercivity:} There exist constants $\coerA>0,\coerB,\coerC\in\R$ and $\alpha_{\mathrm{p}}\in(1,+\infty)$, such that $ u\in \VHilbert$
			\begin{align}\label{growth condition}
				2\prescript{}{\VHilbert'}{\left\langle \Adrift(u),u\right\rangle_{\VHilbert}}
				+
				\norm{\Bvol(u)\sqrt{\Qcor}}
					{ \HilbSchm{\UHilbert}{\HHilbert}  }^2
				\leq
				-\coerA\norm{u}{\VHilbert}^{\alpha_{\mathrm{p}}} + \coerB \norm{u}{\HHilbert}^2 + \coerC.
			\end{align}

			\item \textbf{Boundedness:} There exists a constant $c_{\Adrift}>0$ such that
			\begin{align*}
				\norm{\Adrift(u)}{\VHilbert'}
				\leq c_{\Adrift}\left(1+\norm{u}{\VHilbert}\right)
				~\forall u\in \VHilbert.
			\end{align*}
		\end{enumerate}
	\end{Ass}

	Next let us briefly discuss which (S)PDEs can be treated in this variational framework.\newline
	\begin{Rmk}[Possible Signals]\label{Examples that are}
		As mentioned in the introduction a classical example of a differential operator that satisfies Assumptions \ref{signal assumptions} is the $p$-Laplacian
		\begin{align}\label{p-Laplace}
			\Adrift(v)
			:=\Delta \left(v |v|^{p-1}\right)
		\end{align}
		for any $p\in [1,+\infty)$ on a given bounded domain $\Lambda$ with Dirichlet boundary conditions. In this case $\VHilbert:= \Wspace^{1,p}_{0}(\Lambda)$ is the classical $p$-integrable, first order Sobolev space of functions that vanish on the boundary. The Hilbert space is then set to $\HHilbert:=\Lspace(\Lambda)$. Neumann, or mixed boundary conditions can be treated as well. Thus we can allow for signals that are given by a (noisy) heat or porous media equation. 
		Another differential operator satisfying our assumptions is given by 
		\begin{align*}
			\Adrift(v)
			:=
			-\Delta v - \mathfrak{a} v^3 +  \mathfrak{b} v + \mathfrak{c},
		\end{align*}
		where $\mathfrak{a}\geq 0,~\mathfrak{b},\mathfrak{c}\in\R$, for Dirichlet, Neumann or mixed boundary conditions on suitable domains. Therefore we can treat signals evolving by a stochastic reaction diffusion equation with a double well potential. In particular Allen--Cahn equations can be treated. For a more detailed discussion we refer to \cite[Section 4.1]{RoecknerPrevot}.\newline
	\end{Rmk}
	
	Under Assumption \ref{signal assumptions} it can be shown \cite[Theorem 4.2.4]{RoecknerPrevot} that  
	for a given (random) initial condition $u_0\in \HHilbert$, measurable w.r.t. $\mathfrak{F}_0$ and square-integrable $\mathbb{E}\left[\norm{u_0}{\HHilbert}^2\right]<+\infty$, the signal SPDE \eqref{signal} has a unique strong solution $\left(u_t\right)_{t\geq 0}$, meaning that
	\begin{itemize}
		\item $u$ is adapted to the filtration $\mathfrak{F}$.
		
		\item for any $T<+\infty$ it holds that
		\begin{align*}
			u\in
			\Lspace^2\left([0,T]\times\Omega;\HHilbert\right)
			\cap
			\Lspace^{\alpha_{\mathrm{p}}}\left(
				[0,T]\times\Omega;\VHilbert
			\right)
		\end{align*}
		
		\item for any $t\geq 0$ and any $v\in\VHilbert$, it holds that
		\begin{align*}
			\left\langle u_t,v\right\rangle_{\HHilbert}
			=
			\left\langle u_t,v\right\rangle_{\HHilbert}
			+
			\int_{0}^{t}	
			\prescript{}{\VHilbert'}{\left\langle \Adrift(u_s),v\right\rangle_{\VHilbert}}
			\mathrm{d}s
			+
			\left\langle \int_{0}^{t}\Bvol(u_s)\mathrm{d}W_s,v\right\rangle_{\HHilbert}
			\quad
			\mathbb{P}-a.s.
			.
		\end{align*}
	\end{itemize}
	This unique strong solution $u$ will henceforth be referred to as the signal. In particular one can show that it is an $\HHilbert$-Markov process \cite[Proposition 4.3.5]{RoecknerPrevot} satisfying the following pathwise moment estimate $\mathbb{E}\left[\sup_{t\leq T}\norm{u_t}{\HHilbert}^2\right]<+\infty$ holds for all $T>0$.\newline
	
	An important tool to our analysis is Itô's lemma for variational SPDEs \cite[page 136]{Pardoux}, first derived in \cite{Pardoux - French}. For later reference let us specify here for which functions one can use Itô's lemma.
	
	\begin{Def}\label{Ito function}
		Any function $\phi:\HHilbert\to\R$ is said to be an Itô function, if
		\begin{enumerate}
			\item $\phi$ is twice Fréchet differentiable, with the first two derivatives denoted by $\Frechet{\HHilbert}{\phi}$ and $\Hesse{\HHilbert}{\phi}$.
			
			\item  All of $\phi$, $\Frechet{\HHilbert}{\phi}$ and $\Hesse{\HHilbert}{\phi}$ are locally bounded.
			
			\item For any operator $\Qcor:\HHilbert\to\HHilbert$ that is of trace class, the functional $v\mapsto\mathrm{tr}_{\HHilbert}\left[\Qcor\Hesse{\HHilbert}{\phi}(v)\right]$ is continuous on $\HHilbert$.
			
			\item For any $v\in\VHilbert$ it holds that $\Frechet{\HHilbert}{\phi}(v)\in\VHilbert$ and the map $\restr{\Frechet{\HHilbert}{\phi}}{\VHilbert}:\VHilbert\to\VHilbert$ is continuous when the domain is equipped with the strong and the image is equipped with the weak topology.
			
			\item There exists a constant $C>0$ such that $\norm{\Frechet{\HHilbert}{v}}{\VHilbert}\leq C\left(1+\norm{v}{\VHilbert}\right)$ for all $v\in\VHilbert$.
		\end{enumerate}
		If an Itô function $\phi$ is twice continuously Fréchet differentiable with compact support, we refer to it as an Itô testfunction.
	\end{Def}

	One important example of an Itô function is the squared norm $\norm{.}{\HHilbert}^2$. Furthermore one can use Itô's lemma to show the following basic identities characterizing the distribution of the signal $u$.
	\begin{itemize}
		\item The signal mean $m_t:=\mathbb{E}\left[u_t\right]$ satisfies the differential equation
		\begin{align}\label{equation signal mean}
			\partial_t m_t=\mathbb{E}\left[\Adrift(u_t)\right].
		\end{align}
	
		\item The signal covariance operator $P_t:=\mathbb{Cov}\left[u_t\right]:=\mathbb{E}\left[\left(u_t-m_t\right)\left(u_t-m_t\right)'\right]$ satisfies 
		\begin{align}\label{equation signal covariance}
			\begin{split}
				\partial_t
				\left\langle v,
				\mathbb{Cov}\left[ u_t\right] w\right\rangle_{\HHilbert}
				&=
				\mathbb{E}\left[\left\langle v,u_t-m_t\right\rangle_{\HHilbert}
				\prescript{}{\VHilbert'}{\left\langle \Adrift(u_t)-\Adrift(m_t),w
					\right\rangle_{\VHilbert}}\right]
				\\&\phantom{=}+
				\mathbb{E}\left[\left\langle w,u_t-m_t\right\rangle_{\HHilbert}
				\prescript{}{\VHilbert'}{\left\langle \Adrift(u_t)-\Adrift(m_t),v
					\right\rangle_{\VHilbert}}\right]
				\\&\phantom{=}+
				\left\langle v, \mathbb{E}
				\left[\Bvol(u_t)\sqrt{\Qcor}
				\left(\Bvol(u_t)\sqrt{\Qcor}\right)'
				\right] w
				\right\rangle_{\HHilbert}
			\end{split}
		\end{align}
		for all $v,w\in \VHilbert$.
		
		\item The generator of the signal, denoted by $\mathcal{L}$, is given by
		\begin{align}\label{signal generator}
			\mathcal{L}\phi
			=
			\prescript{}{\VHilbert'}{{\left\langle \Adrift(\cdot),\Frechet{\HHilbert}{\phi}\right\rangle_{\VHilbert}}}
			+
			\frac{1}{2}\mathrm{tr}_{\HHilbert}\left[
			\left(\Hesse{\HHilbert}{\phi}\right)~
			\Bvol(\cdot)\sqrt{\Qcor}\left(\Bvol(\cdot)\sqrt{\Qcor}\right)'
			\right]
		\end{align}
		for every Itô testfunction $\phi$ as per Definition \ref{Ito function}.
	\end{itemize}

	Let us address the observations next. As stated in the introduction we consider continuous, $d_y$-dimensional observations given by the differential equation \eqref{observation}. We make the following standard assumptions for the coefficients $\observ$ and $\Gamma$.\newline
	
	\begin{Ass}[Observation assumptions]\label{Assumption observation}
		The observation function $\observ: \HHilbert\to \R^{d_y}$ is assumed to be globally Lipschitz continuous and $\Gamma\in\Cspace^{0}\left([0,+\infty),\R^{d_y\times d_v}\right)$. Furthermore $V$ is a $d_v$-dimensional, $\mathfrak{F}$-adapted standard Brownian motion, independent of the signal $u$ and its driving noise $W$. As usual we set $R_t:=\Gamma_t\Gamma^{\mathrm{T}}_t$ and assume that $R_t$ is invertible for  all times $t\geq 0$, in particular $\sup_{t\leq T}\left|R_t^{-1}\right|<+\infty$ for all $T<+\infty$.\newline
	\end{Ass}

	In the following, variance bounds for both the posterior \eqref{posterior} and the law of the EnKBF will play a crucial role in our analysis. This is why, besides the standard assumptions for the signal \ref{signal assumptions} and the observations \ref{Assumption observation},  we make the following additional assumption, which will give us a priori bounds for the signal variance.
	
	\begin{Ass}[Bounded signal diffusion]\label{bounded signal diffusion}
		There exists a constant $\boundB<+\infty$ such that
		\begin{align*}
			\sup_{v\in \VHilbert}
			\mathrm{tr}_{\HHilbert}
			\left[\Bvol(v)\sqrt{\Qcor}
			\left(\Bvol(v)\sqrt{\Qcor}\right)'
			\right] 
			\leq\beta.
		\end{align*}
	\end{Ass}
	
	Note that for constant $\Bvol$, i.e. for additive noise in the signal \eqref{signal}, this assumption is fulfilled, as by the Parseval lemma
	\begin{align*}
		\mathrm{tr}_{\HHilbert}
		\left[\Bvol\sqrt{\Qcor}
		\left(\Bvol\sqrt{\Qcor}\right)'
		\right] 
		&=
		\sum_{k\in\N} \left\langle\nu_k\Bvol\Qcor\Bvol'\nu_k\right\rangle_{\HHilbert}
		=
		\sum_{k\in\N}\sum_{l\in\N}
		\QeigValues_l \left\langle\QeigVectors_l,\Bvol'\nu_k\right\rangle_{\HHilbert}^2
		=
		\sum_{l\in\N}
		\QeigValues_l \norm{\Bvol\QeigVectors_l}{\HHilbert}^{2}
		\\&\leq
		\sup_{l\in\N} \QeigValues_l~
		\norm{\Bvol}{\HilbSchm{\UHilbert}{\HHilbert}}^2
		\leq
		\mathrm{tr}\left[\Qcor\right]~\norm{\Bvol}{\HilbSchm{\UHilbert}{\HHilbert}}^2<+\infty.
	\end{align*}

	Finally, to keep formulas short, let us make the following notational convention for conditional expectations that shall hold throughout this paper.
	
	\begin{Def}\label{Definition - conditional expectation abbrev}
		For any point in time $t\geq 0$ and any (sufficiently integrable) $\mathfrak{F}$-adapted, $\HHilbert$-valued process $\left(v_s\right)_{s\geq 0}$ we denote by
		\begin{align*}
			\mathbb{E}_{Y}\left[v_t\right]
			:=
			\mathbb{E}\left[~v_t~\Big|~Y_s,~s\leq t~\right]
		\end{align*}
		its conditional expectation w.r.t. past observations.
	\end{Def}

	\section{Analysis of the EnKBF}\label{section Analysis of particle system}
	
	For the sake of brevity in formulas we make the following definitions.
	\begin{Def}
		For any $\mathfrak{v}=(v_1,\cdots,v_N)\in \HHilbert^{N},~N\in\N$ we set
		\begin{align}\label{Definition empirical moments}
			\begin{split}
				\mathbb{E}^{N}\left[\mathfrak{v}\right]
				&:=
				\frac{1}{N}\sum_{i=1}^{N} v^i
				,~			
				\mathbb{E}^{N}_{\observ}\left[\mathfrak{v}\right]
				:=
				\frac{1}{N}\sum_{i=1}^{N} \observ(v^i)
				\\
				\mathbb{C}^{N}_{\observ}[\mathfrak{v}]
				&:=
				\frac{1}{N}\sum_{i=1}^{N} \left(v^{i}-\mathbb{E}^{N}\left[\mathfrak{v}\right]\right)~ \left(\observ(v^i)-\mathbb{E}^{N}_{\observ}\left[\mathfrak{v}\right]\right)'.
			\end{split}
		\end{align}
		We use the normalization by $1/N$ for the empirical covariance instead of the usual unbiased normalization by $1/(N-1)$ just for notational convenience in the calculations that are to follow.\newline
	\end{Def}
	
	The EnKBF \eqref{EnKBF - introduction} can then be rewritten more compactly into
	\begin{align}\label{nonlinear EnKBF - particle approximation}
		\begin{split}
			\mathrm{d}u^i_t
			&=\Adrift(u^i_t)\mathrm{d}t+\Bvol(u^i_t)\mathrm{d}\bar{W}^{i}_t
			\\&\phantom{=}+
			\mathbb{C}^{N}_{\observ}\left[\mathfrak{u}^N_t\right] R^{-1}_t\left(\mathrm{d}Y_t-\frac{\observ(u^i_t)+\mathbb{E}^{N}_{\observ}\left[\mathfrak{u}^N_t\right]}{2}
			\mathrm{d}t
			\right),
			~\text{for}~i=1,\cdots,N
		\end{split}
	\end{align}
	where $(\bar{W}^i)_{i=1,\cdots,N}$ are independent copies of the Wiener process $W$, driving the signal \eqref{signal}. The initial conditions $u^i_0$ are independent copies of $u_0$ and we denote $\mathfrak{u}^N_t:=(u^1_t,\cdots,u^{N}_t)\in\HHilbert^{N}$.\newline

	While \eqref{nonlinear EnKBF - particle approximation} is just a system of interacting ordinary SPDEs for which local one-sided Lipschitz conditions are enough to derive well posedness \cite{LiuRoeckner-localLipschitzSPDEs}, it does not seem to satisfy the usual growth conditions for unbounded observation functions $\observ$. In \cite{StannatLange} the well posedness was proven for the finite dimensional setting by showing that blow ups do not occur in finite times. This is of course not sufficient to conclude well posedness in infinite dimensions. Instead we employ a partial stopping argument. A priori bounds for the empirical ensemble variance will play a key role in our proof, and so, to keep formulas simple, we make the following definitions<.

	\begin{Def}\label{Definition ensemble variance}
		For any $N\in\N$ and any ensemble $\mathfrak{v}=(v^1,\cdots,v^N)\in \HHilbert^{N}$ of $N$ elements of $\HHilbert$, we define the empirical variance $\sigma^N[\mathfrak{v}]$ by
		\begin{align*}
			\sigma^N[\mathfrak{v}]
			:=
			\frac{1}{N}\sum_{i=1}^{N}\norm{v^i_s-\mathbb{E}^N[\mathfrak{v}]}{\HHilbert}^2.
		\end{align*}
		And similarly we define the empirical observed variance by
		\begin{align*}
			\sigma^{N,\observ}[\mathfrak{v}]
			:=
			\frac{1}{N}\sum_{i=1}^{N}\left|\observ(v^i_s)-\mathbb{E}^N_{\observ}[\mathfrak{v}]\right|^2.
		\end{align*}
	\end{Def}

	\subsection{Well posedness}\label{subsection well posedness of particle EnKBF}
	
	Furthermore we make the following definition, which we will use to switch of parts of the dynamics of the (mean field) EnKBF.
	\begin{Def}\label{Definition - smoothed indicator}
		For any $k\in\N$ we denote by $\tilde{\indicator}_k$ a smoothed version of the indicator function $\indicator_{[0,k]}$, such that $\indicator_{[0,k]}\leq\tilde{\indicator}_k\leq\indicator_{[0,k+1]}$.
	\end{Def}
	
	We are now in the position to formulate and prove the following well posedness result.

	\begin{Thm}\label{Lemma well posedness particle system}
		If the conditions in Assumption \ref{signal assumptions} and \ref{bounded signal diffusion} are satisfied, there exists a unique (strong) solution $\mathfrak{u}^N=\left(u^1,\cdots,u^{N}\right)$ to the nonlinear EnKBF \eqref{nonlinear EnKBF - particle approximation}. I.e. $\mathfrak{u}^N$ is an $\mathfrak{F}$-adapted, $\HHilbert^{N}$-valued process, such that for any $T>0$ and any $i=1,\cdots,N$
		\begin{align*}
			u^i\in
			\Lspace^2\left([0,T]\times\Omega;\HHilbert\right)
			\cap
			\Lspace^{\alpha_{\mathrm{p}}}\left(
			[0,T]\times\Omega;\VHilbert
			\right)
		\end{align*}
		and for any $v\in\VHilbert$
		\begin{align*}
			\left\langle u^i_T, v \right\rangle_{\HHilbert}
			&=
			\left\langle u^i_0, v \right\rangle_{\HHilbert}
			+
			\int_{0}^{T}
			\prescript{}{\VHilbert'}{\left\langle\Adrift(u^i_t),v\right\rangle_{\VHilbert}}
			\mathrm{d}t
			+
			\left\langle\int_{0}^{T}
			\Bvol(u^i_t)\mathrm{d}\bar{W}^{i}_t,v\right\rangle_{\HHilbert}
			\\&\phantom{=}+
			\left\langle
			\int_{0}^{T}
			\tilde{\indicator}_k\left(\norm{\mathbb{C}^{N}_{\observ}\left[\mathfrak{u}^N_t\right]}{\Lspace(\R^{d_y},\HHilbert)}^2\right)
			\mathbb{C}^{N}_{\observ}\left[\mathfrak{u}^N_t\right] R^{-1}_t\left(\mathrm{d}Y_t-\frac{\observ(u^i_t)+\mathbb{E}^{N}_{\observ}\left[\mathfrak{u}^N_t\right]}{2}
			\mathrm{d}t
			\right)
			,v\right\rangle_{\HHilbert}.			
		\end{align*}
	\end{Thm}
	\begin{proof}
		First we note that for any fixed $k\in\N$, the system 
		\begin{align}\label{stopped particle system}
			\begin{split}
				\mathrm{d}u^i_t
				&=\Adrift(u^i_t)\mathrm{d}t+\Bvol(u^i_t)\mathrm{d}\bar{W}^{i}_t
				\\&\phantom{=}+
				\tilde{\indicator}_k\left(\norm{\mathbb{C}^{N}_{\observ}\left[\mathfrak{u}^N_t\right]}{\Lspace(\R^{d_y},\HHilbert)}^2\right)
				\mathbb{C}^{N}_{\observ}\left[\mathfrak{u}^N_t\right] R^{-1}_t\left(\mathrm{d}Y_t-\frac{\observ(u^i_t)+\mathbb{E}^{N}_{\observ}\left[\mathfrak{u}^N_t\right]}{2}
				\mathrm{d}t
				\right),
			\end{split}
		\end{align}
		for $i=1,\cdots,N$, satisfies the standard local one-sided Lipschitz and growth conditions and thus has a unique solution, see \cite{LiuRoeckner-localLipschitzSPDEs}. To keep formulas simple we will omit the argument of $\tilde{\indicator}_k\left(\norm{\mathbb{C}^{N}_{\observ}\left[\mathfrak{u}^N_t\right]}{ \HHilbert^{d_y} }^2\right)$ and instead simply write $\tilde{\indicator}_k$ in the rest of this proof.\newline
		
		We note that the ensemble mean $\mathbb{E}^N\left[\mathfrak{u}^N_t\right]$ satisfies
		\begin{align}\label{evolution empirical mean}
			\mathrm{d}\mathbb{E}^N\left[\mathfrak{u}^N_t\right]
			=
			\frac{1}{N}\sum_{j=1}^{N} \Adrift(u^{j}_t)\mathrm{d}t
			+\frac{1}{N}\sum_{j=1}^{N}\Bvol(u^j_t)\mathrm{d}\bar{W}^{j}_t
			+
			\tilde{\indicator}_k~
			\mathbb{C}^{N}_{\observ}\left[\mathfrak{u}^N_t\right]
			R^{-1}_t\left(
			\mathrm{d}Y_t
			-
			\mathbb{E}^{N}_{\observ}\left[\mathfrak{u}^N_t\right]\mathrm{d}t
			\right).
		\end{align}
		
		This gives us the evolution equation for the centered particles
		\begin{align*}
			\mathrm{d}\left(u^i_t-\mathbb{E}^N\left[\mathfrak{u}^N_t\right]\right)
			&=
			\left(\Adrift(u^{i}_t)-\frac{1}{N}\sum_{j=1}^{N} \Adrift(u^{j}_t)\right)\mathrm{d}t
			+
			\left(\Bvol(u^i_t)\mathrm{d}\bar{W}^{i}_t-\frac{1}{N}\sum_{j=1}^{N}\Bvol(u^j_t)\mathrm{d}\bar{W}^{j}_t\right)
			\\&\phantom{=}-
			\tilde{\indicator}_k~
			\mathbb{C}^{N}_{\observ}\left[\mathfrak{u}^N_t\right] R^{-1}_t\frac{\observ(u^i_t)-\mathbb{E}^{N}_{\observ}\left[\mathfrak{u}^N_t\right]}{2}\mathrm{d}t
			.
		\end{align*}
		
		Note that by Parseval one easily verifies that
		\begin{align*}
			&\frac{1}{N}\sum_{i=1}^{N}
			\left\langle
			u^i_t-\mathbb{E}^N\left[\mathfrak{u}^N_t\right],
			\mathbb{C}^{N}_{\observ}\left[\mathfrak{u}^N_t\right]
			R^{-1}_t\left(\observ(u^i_t)-\mathbb{E}^{N}_{\observ}\left[\mathfrak{u}^N_t\right]\right)
			\right\rangle_{\HHilbert}
			\\&=
			\frac{1}{N}\sum_{i=1}^{N}
			\sum_{k\in\N} 
			\left\langle\nu_k,u^i_t-\mathbb{E}^N\left[\mathfrak{u}^N_t\right]\right\rangle_{\HHilbert}
			\left\langle\nu_k,\mathbb{C}^{N}_{\observ}\left[\mathfrak{u}^N_t\right]
			R^{-1}_t\left(\observ(u^i_t)-\mathbb{E}^{N}_{\observ}\left[\mathfrak{u}^N_t\right]\right)\right\rangle_{\HHilbert}
			\\&=
			\mathrm{tr}_{\HHilbert}\left[\mathbb{C}^{N}_{\observ}\left[\mathfrak{u}^N_t\right]
			R^{-1}_t\mathbb{C}^{N}_{\observ}\left[\mathfrak{u}^N_t\right]^{\mathrm{T}}\right].
		\end{align*}
		
		Therefore, by Itô's formula, we derive the following equation for the average deviation from the ensemble mean
		\begin{align}\label{evolution trace empirical covariance}
			\begin{split}
				\mathrm{d}\sigma^N[\mathfrak{u}^N_t]
				&=
				\frac{2}{N}\sum_{i=1}^{N}
				\prescript{}{\VHilbert'}{\left\langle\left(\Adrift(u^{i}_t)-\frac{1}{N}\sum_{j=1}^{N} \Adrift(u^{j}_t)\right),u^i_t-\mathbb{E}^N\left[\mathfrak{u}^N_t\right]\right\rangle_{\VHilbert}}\mathrm{d}t
				+
				\mathrm{d}\mathfrak{m}^N_t
				\\&\phantom{=}-
				\tilde{\indicator}_k~
				\mathrm{tr}_{\HHilbert}\left[\mathbb{C}^{N}_{\observ}\left[\mathfrak{u}^N_t\right]
				R^{-1}_t\mathbb{C}^{N}_{\observ}\left[\mathfrak{u}^N_t\right]'\right]\mathrm{d}t
				\\&\phantom{=}+
				\frac{1}{N}\sum_{i=1}^{N}
				\mathrm{tr}_{\HHilbert}\left[\Bvol(u^i_t)\sqrt{\Qcor}
				\left(\Bvol(u^i_t)\sqrt{\Qcor}\right)'\right]\mathrm{d}t
				\\&\phantom{=}+
				\frac{1}{N^3}\sum_{i=1}^{N}\sum_{j\neq i}
				\mathrm{tr}_{\HHilbert}\left[\Bvol(u^j_t)\sqrt{\Qcor}
				\left(\Bvol(u^j_t)\sqrt{\Qcor}\right)'\right]\mathrm{d}t
				,
			\end{split}
		\end{align}
		where $\mathfrak{m}^N$ denotes the local martingale given by
		\begin{align}\label{martingale fluctuations of empirical variance}
			\begin{split}
				\mathrm{d}\mathfrak{m}^N_t
				&:=
				\frac{2}{N}\sum_{i=1}^{N}
				\left\langle
				u^i_t-\mathbb{E}^N\left[\mathfrak{u}^N_t\right],
				\left(\Bvol(u^i_t)\mathrm{d}\bar{W}^{i}_t-\frac{1}{N}\sum_{j=1}^{N}\Bvol(u^j_t)\mathrm{d}\bar{W}^{j}_t\right)
				\right\rangle_{\HHilbert}
				\\&=
				\frac{2}{N}\sum_{i=1}^{N}
				\left\langle
				u^i_t-\mathbb{E}^N\left[\mathfrak{u}^N_t\right],
				\Bvol(u^i_t)\mathrm{d}\bar{W}^{i}_t
				\right\rangle_{\HHilbert}
				,~\text{with}~\mathfrak{m}^N_0=0.
			\end{split}
		\end{align}

		We note that in \eqref{evolution trace empirical covariance} we can replace $\frac{1}{N}\sum_{j=1}^{N}  \Adrift(u^{j}_t)$ by $\Adrift\left(\mathbb{E}^N\left[\mathfrak{u}^N_t\right]\right)$. Thus by using the one-sided Lipschitz condition \eqref{one sided Lipschitz} and Assumption \eqref{bounded signal diffusion} as well as the positivity of the trace\\  $\mathrm{tr}_{\HHilbert}\left[\mathbb{C}^{N}_{\observ}\left[\mathfrak{u}^N_t\right]
		R^{-1}_t\mathbb{C}^{N}_{\observ}\left[\mathfrak{u}^N_t\right]'\right]$, we derive the inequality
		\begin{align}\label{variance inequality particle system}
			\mathrm{d}\sigma^N[\mathfrak{u}^N_t]
			\leq
			\left(2\oneSidedLip~\sigma^N[\mathfrak{u}^N_t]+\boundB\right)
			\mathrm{d}t
			+
			\mathrm{d}\mathfrak{m}^N_t.
		\end{align}
		
		Since $\mathfrak{m}^N$ is a real valued local martingale we can deduce by the stochastic Grönwall Lemma \cite[Theorem 4]{Scheutzow} that
		\begin{align}\label{variance inequality particle system - in expectation}
			\mathbb{E}\left[\sup_{t\leq T}\sqrt{\frac{1}{N}\sum_{i=1}^{N}\norm{u^i_t-m^N_t}{\HHilbert}^2}\right]
			=
			\mathbb{E}\left[\sup_{t\leq T}\sqrt{\sigma^N[\mathfrak{u}^N_t]}\right]
			\leq
			(\pi+1)\sqrt{\boundB} e^{\oneSidedLip T}.
		\end{align}
		
		Due to the Lipschitz continuity of $\observ$, this also gives a uniform bound for $\norm{\mathbb{C}^{N}_{\observ}\left[\mathfrak{u}^N_t\right]}{\HHilbert^{d_y}}^2$. Since \eqref{stopped particle system} coincide with \eqref{nonlinear EnKBF - particle approximation} on $\left\{\norm{\mathbb{C}^{N}_{\observ}\left[\mathfrak{u}^N_t\right]}{\HHilbert^{d_y}}^2\leq k\right\}$ we have thus derived the well posedness of \eqref{nonlinear EnKBF - particle approximation}.
		
	\end{proof}

	\begin{Rmk}[Literature]
		As already mentioned, in finite dimensions well posedness of the particle system \eqref{nonlinear EnKBF - particle approximation} was proven in \cite{StannatLange}. An extension of this proof to the correlated noise framework, which requires the control of singular terms can be found in \cite{ErtelStannat}.\\
		In discrete time, EnKFs evolving in infinite dimensional Hilbert spaces were analysed in \cite{ChernovHoelLawNobileTempone} and in the thesis \cite{Kasanicky}.\\
		Finally we mention that the seminal paper \cite{LawKellyStuart} considered the well posedness and accuracy of both discrete and continuous time EnKFs for a class of possibly infinite dimensional signals, which included the 2D Navier-Stokes equation. Existence of strong solutions to the continuous time EnKF \eqref{nonlinear EnKBF - particle approximation} with complete observations ($\observ=\mathrm{id}_{\HHilbert}$) was assumed and it was shown that solutions do not blow up.\newline 
	\end{Rmk}

	\subsection{Exponential moment bounds for the gain term}\label{subsection exponential moment bounds}
	
	The key identity in the proof of Theorem \ref{Lemma well posedness particle system} was \eqref{variance inequality particle system}, which gave a priori bounds for the empirical covariance operator. Later on we will show that the (conditional) covariance of its mean field limit satisfies a similar differential inequality that does not include the random fluctuations caused by the martingale $\mathfrak{m}^N$. If the EnKBF \eqref{nonlinear EnKBF - particle approximation} is to be a good approximation of its mean field limit \eqref{nonlinear EnKBF},  one would expect these fluctations to become small as the number of particles is increased sufficiently. This is indeed the case, as due to Assumption \ref{bounded signal diffusion}, we derive
	\begin{align}\label{empirical variance fluctuation inequality}
		\begin{split}
			\mathrm{d}
			\quadV{\mathfrak{m}^N}_t
			&=
			\frac{2}{N^2}\sum_{i=1}^{N} \left\langle u^i_t-\mathbb{E}^N\left[\mathfrak{u}^N_t\right], \Bvol(u^i_t)\Qcor \Bvol(u^i_t)' (u^i_t-\mathbb{E}^N\left[\mathfrak{u}^N_t\right]) \right\rangle_{\HHilbert}\mathrm{d}t
			\\&\leq
			\frac{2\boundB}{N} \frac{1}{N}\sum_{i=1}^{N} \norm{u^i_t-\mathbb{E}^N\left[\mathfrak{u}^N_t\right]}{\HHilbert}^2
			\mathrm{d}t
			= \frac{2\boundB}{N} \sigma^N[\mathfrak{u}^N_t]\mathrm{d}t.
		\end{split}
	\end{align}
	
	The last inequality is a consequence of the fact that the trace is invariant under the change of the orthonormal basis and that for every nonzero vector one can find an orthonormal basis that contains this vector. Since we were able to bound $\sigma^N[\mathfrak{u}^N_t]$ uniformly in time, the quadratic variation of $\mathfrak{m}^N$ will decrease to zero for $N\to\infty$.
	For our convergence proof in section \ref{section prop of chaos} we will need exponential moment bounds of the empirical variance $\sigma^N[\mathfrak{u}^N_t]$. Such bounds are a delicate matter as the ensemble $\mathfrak{u}^N$ will likely show some Gaussian (tail) behaviour and thus $\mathbb{E}\left[\sup_{t\leq T}\exp\left(r\sigma^N[\mathfrak{u}^N_t]\right)\right]$ might not be finite for all values of $r\geq 0$. However, as $N\to\infty$ one would expect $\sigma^N$ to become deterministic and as such any exponential moment should exist for $N$ sufficiently large. We prove this fact in the following lemma by employing a Grönwall argument.

	\begin{Prop}\label{Lemma exponential moments} 
		Let $q\geq0$ be arbitrary. Then for any $N\in\N$ such that $N>2\boundB q~e^{(2\oneSidedLip+1)T}$ we have
		\begin{align}
			\mathbb{E}\left[\sup_{t\leq T}
			\exp\left(q~\sigma^{N}\left[\mathfrak{u}^N_t\right]\right)\right]
			\leq
			(\pi+1) 
			\exp\left(\frac{q\left(e^{(2\oneSidedLip+1)T}-1\right)}{2(2\oneSidedLip+1)}\right)
			\mathbb{E}\left[\exp\left(2 q e^{(2\oneSidedLip+1)T} \sigma^N\left[\mathfrak{u}^N_0\right]\right)\right].
		\end{align}
		In particular the $q$-th exponential moment of the path of $\sigma^N[\mathfrak{u}^N]$ exists up to time $T$, if the $\left(2 q e^{(2\oneSidedLip+1)T}\right)$-th exponential moment of the initial empirical variance $\sigma^N[\mathfrak{u}^N]$ exists.
	\end{Prop}
	\begin{proof}
		Let $\mathfrak{a}:=2\oneSidedLip+1$ (see Assumption \ref{signal assumptions}) and $\mathfrak{b}:=q e^{\mathfrak{a}T}$. We define the process $\mathfrak{s}_t:=2\mathfrak{b}~e^{-\mathfrak{a} t } \sigma^{N}[\mathfrak{u}^{N}_t]$. Then, using inequality \eqref{variance inequality particle system}, we derive the inequality
		\begin{align*}
			\mathrm{d}\mathfrak{s}_t
			=
			2\mathfrak{b}e^{-\mathfrak{a} t } \mathrm{d}\sigma^{N}[\mathfrak{u}^N_t]
			-2\mathfrak{a}\mathfrak{b}e^{-\mathfrak{a}t}\sigma^{N}[\mathfrak{u}^N_t]\mathrm{d}t
			&\leq
			(2\oneSidedLip-\mathfrak{a})2\mathfrak{b} e^{-\mathfrak{a}}\sigma^{N}\left[\mathfrak{u}^N_t\right]\mathrm{d}t
			+2\mathfrak{b}\boundB e^{-\mathfrak{a}}\mathrm{d}t
			+2\mathfrak{b}e^{-\mathfrak{a}t}\mathrm{d}\mathfrak{m}_t
			\\&=
			(2\oneSidedLip-\mathfrak{a}) \mathfrak{s}_t\mathrm{d}t
			+2\mathfrak{b}\boundB e^{-\mathfrak{a}}\mathrm{d}t
			+2\mathfrak{b}e^{-\mathfrak{a}t}\mathrm{d}\mathfrak{m}_t.
		\end{align*}
		
		Furthermore we derive from \eqref{evolution trace empirical covariance} the form of the quadratic variation of $\mathfrak{s}$, and from  \eqref{empirical variance fluctuation inequality} the estimate
		\begin{align*}
			\mathrm{d}\quadV{\mathfrak{s}}_t
			=
			(2\mathfrak{b})^2 e^{-2\mathfrak{a}t}
			\mathrm{d}\quadV{\mathfrak{m}^N}_t
			\leq
			(2\mathfrak{b})^2 e^{-2\mathfrak{a}t}
			\frac{2\boundB}{N}\sigma^N[\mathfrak{u}^N_t]\mathrm{d}t
			=
			2\frac{2\boundB\mathfrak{b}e^{-\mathfrak{a}t}}{N}\mathfrak{s}_t
			\mathrm{d}t
			.
		\end{align*}
		
		These inequalities, together with Itô's formula, give us the following inequality
		\begin{align*}
			\mathrm{d}\exp(\mathfrak{s}_t)
			&=
			\exp(\mathfrak{s}_t)\mathrm{d}\mathfrak{s}_t
			+\exp(\mathfrak{s}_t)
			\frac{1}{2}\mathrm{d}\quadV{\mathfrak{s}}_t
			\\&\leq
			(2\oneSidedLip-\mathfrak{a}) \mathfrak{s}_t\exp(\mathfrak{s}_t)\mathrm{d}t
			+
			2\mathfrak{b}\boundB e^{-\mathfrak{a}t}\exp(\mathfrak{s}_t)\mathrm{d}t
			+
			2\mathfrak{b}e^{-\mathfrak{a}t}\exp(\mathfrak{s}_t)\mathrm{d}\mathfrak{m}_t
			+
			\frac{2\boundB\mathfrak{b}e^{-2\mathfrak{a}t}}{N}\mathfrak{s}_t
			\exp(\mathfrak{s}_t)
			\mathrm{d}t
			\\&=
			\left(2\oneSidedLip-\mathfrak{a}+\frac{2\boundB\mathfrak{b}e^{-\mathfrak{a}t}}{N}\right) \mathfrak{s}_t\exp(\mathfrak{s}_t)\mathrm{d}t
			+
			2\mathfrak{b}\boundB e^{-\mathfrak{a}t}\exp(\mathfrak{s}_t)\mathrm{d}t
			+
			2\mathfrak{b}e^{-\mathfrak{a}t}\exp(\mathfrak{s}_t)\mathrm{d}\mathfrak{m}_t
			.
		\end{align*}
		
		Due to our assumptions we have $\mathfrak{a}>2\oneSidedLip+\frac{2\boundB\mathfrak{b}e^{-\mathfrak{a}t}}{N}$ and thus derive the stochastic inequality
		\begin{align*}
			\mathrm{d}\exp(\mathfrak{s}_t)
			\leq
			2\mathfrak{b}\boundB e^{-\mathfrak{a}t}\exp(\mathfrak{s}_t)\mathrm{d}t
			+
			2\mathfrak{b}e^{-\mathfrak{a}t}\exp(\mathfrak{s}_t)\mathrm{d}\mathfrak{m}_t.
		\end{align*}
		
		Since $2\mathfrak{b}e^{-\mathfrak{a}t}\exp(\mathfrak{s}_t)\mathrm{d}\mathfrak{m}_t$ defines a local martingale, the stochastic Grönwall inequality \cite{Scheutzow}  gives us
		\begin{align*}
			\mathbb{E}\left[\sup_{t\leq T}
			\exp\left(q~\sigma^{N}\left[\mathfrak{u}^N_t\right]\right)\right]
			&\leq
			\mathbb{E}\left[\sup_{t\leq T}
			\sqrt{\exp\left(\mathfrak{s}_t\right)}\right]
			\leq
			(\pi+1) \exp\left(q/2 e^{\mathfrak{a}T}\int_{0}^{T}e^{-\mathfrak{a} s}\mathrm{d}s\right)
			\mathbb{E}\left[
			\exp\left(\mathfrak{s}_0/2\right)\right]
			\\&\leq
			(\pi+1) 
			\exp\left(\frac{q\left(e^{(2\oneSidedLip+1)T}-1\right)}{2(2\oneSidedLip+1)}\right)
			\mathbb{E}\left[\exp\left(2 q e^{(2\oneSidedLip+1)T} \sigma^N\left[\mathfrak{u}^N_0\right]\right)\right],
		\end{align*}
		which concludes the proof.
	\end{proof}
	
	\begin{Rmk}
		In the proof of Proposition \ref{Lemma exponential moments} we used a standard testfunction for our Grönwall argument. Since we have good controls for the quadratic variation of $\mathfrak{m}$, we also could have just used the standard Burkholder--Davis--Gundy inequality in combination with a deterministic Grönwall Lemma. The usage of the stochastic Grönwall inequality is not necessary in our setting, however in \cite{HuddeHutzenthalerMazzonetto} a similar testfunction and a novel stochastic Grönwall--Lyapunov inequality were used to derive uniform exponential moment bounds for SDEs satisfying an appropriate Lyapunov condition.
	\end{Rmk}

	\section{Analysis of the mean field EnKBF}\label{section mf EnKBF}

	Using the notational convention for conditional expectations of Definition \ref{Definition - conditional expectation abbrev}, the mean field EnKBF \eqref{mean field EnKBF - introduction} can be written more compactly as
	\begin{align}\label{nonlinear EnKBF}
		\begin{split}
			\mathrm{d}\bar{u}_t
			&=\Adrift(\bar{u}_t)\mathrm{d}t+\Bvol(\bar{u}_t)\mathrm{d}\bar{W}_t
			+
			\mathbb{Cov}_Y\left[\bar{u}_t,\observ(\bar{u}_t)\right] R^{-1}_t\left(\mathrm{d}Y_t-\frac{\observ(\bar{u}_t)+\mathbb{E}_Y\left[\observ(\bar{u}_t)\right]}{2}
			\mathrm{d}t
			\right)
			.
		\end{split}
	\end{align}
	When taking the (formal) mean field limit $N\to\infty$ for the EnKBF \eqref{nonlinear EnKBF - particle approximation}, one would expect it to converge in an appropriate sense to \eqref{nonlinear EnKBF}. Before we show this convergence rigorously in the next section \ref{section prop of chaos}, let us first investigate the well posedness of the McKean--Vlasov equation \eqref{nonlinear EnKBF}.\newline

	\begin{Rmk}[Literature]
		The mean field EnKBF \eqref{nonlinear EnKBF} does not seem to fit the standard well posedness theory for McKean--Vlasov equations, as found e.g. in \cite{CarmonaDelarue}, due to only locally Lipschitz coefficients and missing growth conditions. In a finite dimensional setting well posedness of the mean field EnKBF \eqref{nonlinear EnKBF} was shown in \cite{Coghi et al} for bounded signal dynamics and observation functions. For linear observation functions \cite{ErtelStannat} showed well posedness of finite dimensional mean-field EnKBFs that may also include singular correction terms in the presence of correlated noise. This was done by a combination of a fixed point and a stopping argument with respect to the covariance $\mathbb{Cov}_Y\left[\bar{u}_t,\observ(\bar{u}_t)\right]$. The main tool were a priori variance bounds, that were made robust with respect to the fixed point argument via stopping times.\\
		In the infinite dimensional setting this argument does not work due to missing equivalence of norms.
		While using a Galerkin argument would thus seem tempting, it would also not imply the desired uniqueness of solutions, which is  a property that is difficult to show, and sometimes does not even hold for McKean--Vlasov equations under local Lipschitz conditions \cite{Scheutzow}.\\		
		So instead we use an adapted  fixed point argument, that makes use of variance inequalities, similar to those used in the proof of Theorem \ref{Lemma well posedness particle system}.
	\end{Rmk}

	First we investigate the covariance structure in \eqref{nonlinear EnKBF}. We do this however in a more general form. Let $(h_t)_{t\geq  0}$ be a given $R^{d_y}$-valued, $\mathfrak{F}$-adapted stochastic process with finite second moments. Furthermore assume that $\xi^Y$ is an  $\R^{d_y}$-valued semimartingale that is adapted to the natural filtration generated by $Y$, and thus also to $\mathfrak{F}$. Then for a $\tilde{u}$ satisfying	
	\begin{align}\label{general observation function structure}
		\begin{split}
			\mathrm{d}\tilde{u}_t
			&=\Adrift(\tilde{u}_t)\mathrm{d}t+\Bvol(\tilde{u}_t)\mathrm{d}\bar{W}_t
			+
			\mathbb{Cov}_Y\left[\tilde{u}_t,h_t\right] R^{-1}_t\left(\mathrm{d}\xi^Y_t-\frac{h_t+\mathbb{E}_Y\left[h_t\right]}{2}
			\mathrm{d}t
			\right)
		\end{split}
	\end{align}
	one can easily use Itô's formula to show that
	\begin{align*}
		\begin{split}
			&\partial_t
			\left\langle v,
			\mathbb{Cov}_Y\left[\tilde{u}_t\right] w\right\rangle_{\HHilbert}
			\\&=
			\mathbb{E}_Y\left[
			\left\langle v,\tilde{u}_t-\tilde{m}_t\right\rangle_{\HHilbert}
			\prescript{}{\VHilbert'}{\left\langle w,\Adrift(\tilde{u}_t)-\Adrift(\tilde{m}_t)
				\right\rangle_{\VHilbert}}
			+
			\left\langle w,\tilde{u}_t-\tilde{m}_t\right\rangle_{\HHilbert}
			\prescript{}{\VHilbert'}{\left\langle v,\Adrift(\tilde{u}_t)-\Adrift(\tilde{m}_t)
				\right\rangle_{\VHilbert}}
			\right]\mathrm{d}t
			\\&\phantom{=}+
			\mathbb{E}_Y\left[\left\langle v, \Bvol(\tilde{u}_t)\sqrt{\Qcor}
			\left(\Bvol(\tilde{u}_t)\sqrt{\Qcor}\right)' w
			\right\rangle_{\HHilbert}\right]
			-
			\left\langle v,
			\mathbb{Cov}_Y\left[\tilde{u}_t,h_t\right] R^{-1}_t \mathbb{Cov}_Y\left[h_t,\tilde{u}_t\right] w\right\rangle_{\HHilbert}
			.
		\end{split}
	\end{align*}
	for every $v,w\in \VHilbert$. Thus by the positivity of $\left\langle v,
	\mathbb{Cov}_Y\left[\tilde{u}_t,h_t\right] R^{-1}_t \mathbb{Cov}_Y\left[h_t,\tilde{u}_t\right] v\right\rangle_{\HHilbert}$ for every $v\in \VHilbert$ we immediately derive
	\begin{align}\label{law of total variance EnKBF}
		\partial_t\mathrm{tr}_{\HHilbert}\mathbb{Cov}_Y\left[\tilde{u}_t\right]
		\leq
		\oneSidedLip~
		\mathrm{tr}_{\HHilbert}\mathbb{Cov}_Y\left[\tilde{u}_t\right]
		+
		\boundB,
	\end{align}
	and therefore
	\begin{align}\label{variance bound tilde u}
		\mathbb{E}_Y\left[\norm{\tilde{u}_t-\mathbb{E}_Y\left[\tilde{u}_t\right]}{\HHilbert}^2\right]
		=
		\mathrm{tr}_{\HHilbert}\mathbb{Cov}_Y\left[\tilde{u}_t\right]
		\leq
		\boundB e^{\oneSidedLip t}.
	\end{align}
	
	{Thus the EnKBF satisfies the variance bound that is implied by the Bayesian filtering problem (see Appendix \ref{section Kushner Stratonovich}), it does so even in a stronger sense, as taking the expectation is not required. As implied by the law of total variance for the optimal filter, this bound is robust with respect to perturbations of both the modelled observation function $\observ$ and the actual observation data $Y$.\newline}
	
	Next we show that the robust variance bound \eqref{variance bound tilde u} can be used to show well posedness of the EnKBF via a Picard argument.\newline
	
	\begin{Thm}\label{well posedness mf EnKBF}
		If the conditions in Assumption \ref{signal assumptions} and \ref{bounded signal diffusion} are satisfied, there exists a unique strong solution to the nonlinear mean-field EnKBF \eqref{nonlinear EnKBF}\footnote{Hereby the notion of strong solution is defined just as for the particle system \eqref{nonlinear EnKBF - particle approximation} in Theorem \ref{Lemma well posedness particle system}.}.
	\end{Thm}
	\begin{proof}
		For proving well posedness it is enough to restrict ourselves to a small time frame $[0,T]$ with $T$ chosen later on. The extension to arbitrary time frames can then be easily achieved by standard glueing arguments\newline
		
		The proof is separated into two steps. First we introduce partially stopped dynamics and show their well posedness via a fixed point argument. Next we show that these stopped dynamics must always coincide with solutions to the EnKBF on events that cover the whole probability space almost surely.\newline
		
		In the following we will make use of the semimartingale decomposition of the observation process $Y$. For highlighting that the true signal process $u$, given by \eqref{signal}, plays the role of a parameter to the EnKBF and to easily distinguish it from other processes encountered in the proof we will denote it by $u^{\mathrm{ref}}$. The observation process $Y$ is thus given by
		\begin{align}\label{specific form of observations}
			\mathrm{d}Y_t=\observ(u^{\mathrm{ref}}_t)\mathrm{d}t+\Gamma_t\mathrm{d}V_t.
		\end{align}

		\uline{Step 1:} \textit{Well posedness for partially stopped dynamics.}\newline

		Define for $k,l\in\N$ and  any $\HHilbert$-valued random variable $v$ the stopped observation function $\observ^k$ by
		\begin{align*}
			\observ^{k}(v):= \tilde{\indicator}_{k}\left(\left|\mathbb{E}_{Y}\left[\observ(v)\right]\right|\right)
			\observ(v),
		\end{align*}
		where $\tilde{\indicator}_{k}$ is the smoothed indicator function of Definition \ref{Definition - smoothed indicator}. Furthermore define the stopped observation process $Y^{l}$ by
		\begin{align*}
			\mathrm{d}Y^{l}_t
			:=\tilde{\indicator}_{l}\left(\mathbb{E}_Y\left[\left|\observ(u^{\mathrm{ref}}_t)\right|^2\right]\right)
			\mathrm{d}Y_t.
		\end{align*}

		In this step we show that there exists a unique solution $\bar{u}^k$ of the (partially) stopped dynamics
		\begin{align}\label{partially stopped mf EnKBF}
			\begin{split}
				\mathrm{d}\bar{u}^k_t
				&=\Adrift(\bar{u}^k_t)\mathrm{d}t+\Bvol(\bar{u}^k_t)\mathrm{d}\bar{W}_t
				\\&\phantom{=}+
				\mathbb{Cov}_Y\left[\bar{u}^k_t,\observ^{k}\left(\bar{u}^k_t\right)\right] R^{-1}_t\left(\mathrm{d}Y^{k}_t-
				\frac{\observ^{k}\left(\bar{u}^k_t\right)+\mathbb{E}_Y\left[\observ^{k}\left(\bar{u}^k_t\right)\right]}{2}
				\mathrm{d}t
				\right),
			\end{split}
		\end{align}
		via a fixed point argument with respect to the stopped modelled observations $\left(\observ(\bar{u}^k_t)\right)_{t\in[0,T]}$.\newline
		
		To this end we consider for a given process $h$, which is assumed to be an $\mathfrak{F}$-adapted, square integrable semimartingale, the unique solution $\tilde{u}$ of
		\begin{align}\label{fixed point equation stopped dynamics}
			\begin{split}
				\mathrm{d}\tilde{u}_t
				&=\Adrift(\tilde{u}_t)\mathrm{d}t+\Bvol(\tilde{u}_t)\mathrm{d}\bar{W}_t
				\\&\phantom{=}+
				\tilde{\indicator}_{k}\left(\left|\mathbb{E}_Y\left[h_t\right]\right|\right)
				\mathbb{Cov}_Y\left[\tilde{u}_t,h_t\right] R^{-1}_t\left(\mathrm{d}Y^{l}_t-
				\tilde{\indicator}_{k}\left(\left|\mathbb{E}_Y\left[h_t\right]\right|\right)
				\frac{h_t+\mathbb{E}_Y\left[h_t\right]}{2}
				\mathrm{d}t
				\right).
			\end{split}
		\end{align}
		Well posedness of \eqref{fixed point equation stopped dynamics} is assured by the standard (global) Lipschitz and growth conditions, which assure that a unique solution can be found by standard Picard fixed point arguments as found for example in \cite{CarmonaDelarue} for finite dimensional McKean--Vlasov equations\footnote{One can also use the approach of \cite{Liu}, which showed well posedness of McKean--Vlasov SPDEs, but only consider deterministic coefficients.}.  We define the map $\Xi$ by
		\begin{align*}
			\Xi(h):=\observ(\tilde{u}).
		\end{align*}
		Since $\observ(\tilde{u})$ is an $\mathfrak{F}$-adapted, square integrable semimartingale, the existence and uniqueness of solutions to \eqref{partially stopped mf EnKBF} corresponds to the existence and uniqueness of fixed points of $\Xi$. We prove this via a Banach fixed point argument and thus have to show the contractivity of $\Xi$. Due to the Lipschitz continuity of $\observ$, this further reduces to the problem of showing that the solution map $h\mapsto \tilde{u}$ defined by the equation \eqref{general observation function structure} is Lipschitz, with constant strictly smaller than $1/\mathrm{Lip}(\observ)$.\newline
		
		Since \eqref{fixed point equation stopped dynamics} is of the form \eqref{general observation function structure}, the process $\tilde{u}$ must also satisfy the uniform variance bound \eqref{variance bound tilde u} corresponding to the law of total variance. Therefore, by the  Lipschitz  continuity of $\observ$, we can assume that any potential fixed point $h$ satisfies
		\begin{align}\label{variance bound potential fixed point}
			\mathbb{Var}_Y\left[h_t\right]
			=\mathrm{tr}_{\R^{d_y}}\mathbb{Cov}_Y\left[h_t\right]
			\leq \mathrm{Lip}(\HHilbert) \boundB e^{\oneSidedLip t}.
		\end{align}
		
		To show the contractivity of $\Xi$, let $h^i,~i=1,2$ be two given processes and denote by $\tilde{u}^i,~i=1,2$ the corresponding solutions to \eqref{fixed point equation stopped dynamics}. Using the uniform variance bounds \eqref{variance bound potential fixed point}, as well as the Lipschitz continuity of $\tilde{\indicator}_k$ and its boundedness $0\leq\tilde{\indicator}_k\leq 1$, we derive the following bound for the gain difference
		\begin{align}\label{covariance difference}
			\begin{split}
				&\norm{
					\tilde{\indicator}_{k}\left(\left|\mathbb{E}_Y\left[h^1_t\right]\right|\right)
					\mathbb{Cov}_Y\left[\tilde{u}^1_t,h^1_t\right]
					-
					\tilde{\indicator}_{k}\left(\left|\mathbb{E}_Y\left[h^2_t\right]\right|\right)
					\mathbb{Cov}_Y\left[\tilde{u}^2_t,h^2_t\right]}{L(\R^{d_y},\HHilbert)}
				\\&\leq
				\left|\tilde{\indicator}_{k}\left(\left|\mathbb{E}_Y\left[h^1_t\right]\right|\right)-\tilde{\indicator}_{k}\left(\left|\mathbb{E}_Y\left[h^2_t\right]\right|\right)\right|~
				\norm{
					\mathbb{Cov}_Y\left[\tilde{u}^1_t,h^1_t\right]
				}{L(\R^{d_y},\HHilbert)}
				\\&\phantom{=}+
				\sqrt{\mathbb{E}_Y\left[\norm{\tilde{u}^1_s-\tilde{u}^2_s}{\HHilbert}^2\right]}
				\sqrt{\mathbb{E}_Y\left[\norm{h^1_s-\mathbb{E}_Y\left[h^1_s\right]}{\R^{d_y}}^2\right]}
				\\&\phantom{=}+
				\sqrt{\mathbb{E}_Y\left[\norm{\tilde{u}^2_t-\mathbb{E}_Y\left[\tilde{u}^2_t\right]}{\HHilbert}^2\right]}
				\sqrt{\mathbb{E}_Y\left[\norm{h^1_t-h^2_t}{\R^{d_y}}^2\right]}
				\\&\leq
				\left(\left(\mathrm{Lip}(\tilde{\indicator}_k)+1\right)\sqrt{\mathrm{Lip}(\observ)}+1\right)\boundB e^{\oneSidedLip t}
				\left(
				\sqrt{\mathbb{E}_Y\left[\norm{\tilde{u}^1_s-\tilde{u}^2_s}{\HHilbert}^2\right]}
				+
				\sqrt{\mathbb{E}_Y\left[\norm{h^1_s-h^2_s}{\R^{d_y}}^2\right]}
				\right).
			\end{split}
		\end{align}

		Using Itô's formula for the squared norm, we derive
		\begin{align*}
			&\norm{\tilde{u}^1_t-\tilde{u}^2_t}{\HHilbert}^2
			=
			2\int_{0}^{t}
			\prescript{}{\VHilbert'}{\left\langle \Adrift(\tilde{u}^1_s)-\Adrift(\tilde{u}^2_s),\tilde{u}^1_s-\tilde{u}^2_s\right\rangle_{\VHilbert}}\mathrm{d}s
			\\&+
			2\int_{0}^{t}
			\left\langle \tilde{u}^1_s-\tilde{u}^2_s,
			\left(\tilde{\indicator}_{k}\left(\left|\mathbb{E}_Y\left[h^1_t\right]\right|\right)
			\mathbb{Cov}_Y\left[\tilde{u}^1_s,h^1_s\right]
			-
			\tilde{\indicator}_{k}\left(\left|\mathbb{E}_Y\left[h^2_t\right]\right|\right)
			\mathbb{Cov}_Y\left[\tilde{u}^2_s,h^2_s\right]\right) R^{-1}_s\mathrm{d}Y^k_s\right\rangle_{\HHilbert}
			\\&-
			2\int_{0}^{t}
			\left\langle \tilde{u}^1_s-\tilde{u}^2_s,
			\left(
			\tilde{\indicator}_{k}\left(\left|\mathbb{E}_Y\left[h^1_s\right]\right|\right)
			\mathbb{Cov}_Y\left[\tilde{u}^1_s,h^1_s\right]
			-
			\tilde{\indicator}_{k}\left(\left|\mathbb{E}_Y\left[h^2_s\right]\right|\right)
			\mathbb{Cov}_Y\left[\tilde{u}^2_s,h^2_s\right]\right)
			\right.
			\\&\left.\phantom{2\int_{0}^{t}\langle u^1-u^2,++}
			R^{-1}_s
			\tilde{\indicator}_{k}\left(\left|\mathbb{E}_Y\left[h^1_s\right]\right|\right)
			\frac{h^1_s+\mathbb{E}_Y\left[h^1_s\right]}{2}
			\right\rangle_{\HHilbert}
			\mathrm{d}s
			\\&-
			2\int_{0}^{t}
			\left\langle \tilde{u}^1_s-\tilde{u}^2_s,
			\tilde{\indicator}_{k}\left(\left|\mathbb{E}_Y\left[h^2_s\right]\right|\right)
			\mathbb{Cov}_Y\left[\tilde{u}^2_s,h^2_s\right] R^{-1}_s  
			\right.
			\\&\left.\phantom{2\int_{0}^{t}\langle u^1-u^2,++}
			\frac{
				\tilde{\indicator}_{k}\left(\left|\mathbb{E}_Y\left[h^1_s\right]\right|\right)
				\left(h^1_s+\mathbb{E}_Y\left[h^1_s\right]\right)
				-
				\tilde{\indicator}_{k}\left(\left|\mathbb{E}_Y\left[h^2_s\right]\right|\right)
				\left(h^2_s+\mathbb{E}_Y\left[h^2_s\right]\right)
			}{2}
			\right\rangle_{\HHilbert}
			\mathrm{d}s
			\\&+
			2\int_{0}^{t}
			\left\langle \tilde{u}^1_s-\tilde{u}^2_s,\left(\Bvol(\tilde{u}^1_s)-\Bvol(\tilde{u}^2_s)\right)\mathrm{d}W_s
			\right\rangle_{\HHilbert}
			\\&+
			\sum_{k\in\N}
			\int_{0}^{t}
			\left\langle \nu_k,
			\left(
			\tilde{\indicator}_{k}\left(\left|\mathbb{E}_Y\left[h^1_s\right]\right|\right)
			\mathbb{Cov}_Y\left[\tilde{u}^1_s,h^1_s\right]
			-
			\tilde{\indicator}_{k}\left(\left|\mathbb{E}_Y\left[h^2_s\right]\right|\right)
			\mathbb{Cov}_Y\left[\tilde{u}^2_s,h^2_s\right]\right) \right\rangle_{\HHilbert}
			R^{-1}_s
			\\&\phantom{=\sum_{k\in\N}
				\int_{0}^{t}}
			\left\langle \nu_k,
			\left(
			\tilde{\indicator}_{k}\left(\left|\mathbb{E}_Y\left[h^1_s\right]\right|\right)
			\mathbb{Cov}_Y\left[\tilde{u}^1_s,h^1_s\right]
			-
			\tilde{\indicator}_{k}\left(\left|\mathbb{E}_Y\left[h^2_s\right]\right|\right)
			\mathbb{Cov}_Y\left[\tilde{u}^2_s,h^2_s\right]\right) \right\rangle_{\HHilbert}^{\mathrm{T}}
			\mathrm{d}s
			\\&+
			\sum_{k\in\N}\sum_{n\in\N}
			\int_{0}^{t}
			\QeigValues_n
			\left\langle
			\nu_k,(\Bvol(\tilde{u}^1_s)-\Bvol(\tilde{u}^2_s))\QeigVectors_n
			\right\rangle_{\HHilbert}^2
			\mathrm{d}s.
		\end{align*}
		
		Now we note that by Parseval and the onesided Lipschitz condition \eqref{one sided Lipschitz} we have
		\begin{align}\label{proof obs continuity - inequ Ito correction}
			\begin{split}
				&\sum_{k\in\N}\sum_{n\in\N}
				\QeigValues_n
				\left\langle
				\nu_k,(\Bvol(\tilde{u}^1_s)-\Bvol(\tilde{u}^2_s))\QeigVectors_n
				\right\rangle_{\HHilbert}^2
				=
				\norm{\left(\Bvol(u)-\Bvol(v)\right)\circ\sqrt{\Qcor}}{ \HilbSchm{\UHilbert}{\HHilbert} }^2
				\leq
				\oneSidedLip~\norm{\tilde{u}^1_s-\tilde{u}^2_s}{\HHilbert}^2,
			\end{split}
		\end{align}
		as well as
		\begin{align}\label{proof obs continuity - inequ obsIto}
			\begin{split}
				&\sum_{k\in\N}
				\left\langle \nu_k,
				\left(
				\tilde{\indicator}_{k}\left(\left|\mathbb{E}_Y\left[h^1_s\right]\right|\right)\mathbb{Cov}_Y\left[\tilde{u}^1_s,h^1_s\right]
				-
				\tilde{\indicator}_{k}\left(\left|\mathbb{E}_Y\left[h^2_s\right]\right|\right)
				\mathbb{Cov}_Y\left[\tilde{u}^2_s,h^2_s\right]\right) \right\rangle_{\HHilbert}
				R^{-1}_s
				\\&\phantom{=\times} 
				\left\langle \nu_k,
				\left(
				\tilde{\indicator}_{k}\left(\left|\mathbb{E}_Y\left[h^1_s\right]\right|\right)
				\mathbb{Cov}_Y\left[\tilde{u}^1_s,h^1_s\right]
				-
				\tilde{\indicator}_{k}\left(\left|\mathbb{E}_Y\left[h^2_s\right]\right|\right)
				\mathbb{Cov}_Y\left[\tilde{u}^2_s,h^2_s\right]\right) \right\rangle_{\HHilbert}^{\mathrm{T}}
				\\&=
				\norm{\left(
					\tilde{\indicator}_{k}\left(\left|\mathbb{E}_Y\left[h^1_s\right]\right|\right)
					\mathbb{Cov}_Y\left[\tilde{u}^1_s,h^1_s\right]
					-
					\tilde{\indicator}_{k}\left(\left|\mathbb{E}_Y\left[h^2_s\right]\right|\right)
					\mathbb{Cov}_Y\left[\tilde{u}^2_s,h^2_s\right]\right)R^{-1/2}_s}{L(\R^{d_y},\HHilbert)}^2
				\\&\leq
				\left(\left(\mathrm{Lip}(\tilde{\indicator}_k)+1\right)\sqrt{\mathrm{Lip}(\observ)}+1\right)^2\boundB^2 e^{^2\oneSidedLip t} \left|R^{-1/2}_s\right|^2
				\left(\mathbb{E}_Y\left[\norm{\tilde{u}^1_s-\tilde{u}^2_s}{\HHilbert}^2\right]
				+
				\mathbb{E}_Y\left[\norm{h^1_s-h^2_s}{\R^{d_y}}^2\right]\right).
			\end{split}
		\end{align}
		
		Furthermore we note that
		\begin{align*}
			&\tilde{\indicator}_{k}\left(\left|\mathbb{E}_Y\left[h^2_s\right]\right|\right)
			\mathbb{E}_Y\left[
			\left|
			\tilde{\indicator}_{k}\left(\left|\mathbb{E}_Y\left[h^1_s\right]\right|\right)
			h^1_s
			-
			\tilde{\indicator}_{k}\left(\left|\mathbb{E}_Y\left[h^2_s\right]\right|\right)
			h^2_s
			\right|^2
			\right]
			\\&\leq
			2\tilde{\indicator}_{k}\left(\left|\mathbb{E}_Y\left[h^2_s\right]\right|\right)
			\tilde{\indicator}_{k}\left(\left|\mathbb{E}_Y\left[h^1_s\right]\right|\right)
			\mathbb{E}_Y\left[\left|h^1_s-h^2_s\right|^2\right]
			\\&\phantom{=}+
			2
			\tilde{\indicator}_{k}\left(\left|\mathbb{E}_Y\left[h^2_s\right]\right|\right)\mathbb{E}_Y\left[\left|h^2_s\right|^2\right]
			\left|\tilde{\indicator}_{k}\left(\left|\mathbb{E}_Y\left[h^1_s\right]\right|\right)
			-
			\tilde{\indicator}_{k}\left(\left|\mathbb{E}_Y\left[h^2_s\right]\right|\right)
			\right|^2
			\\&\leq 
			2\mathbb{E}_Y\left[\left|h^1_s-h^2_s\right|^2\right]
			+
			2
			\tilde{\indicator}_{k}\left(\left|\mathbb{E}_Y\left[h^2_s\right]\right|\right)
			\left(
			\left|\mathbb{E}_Y\left[h^2_s\right]\right|^2
			+
			\mathbb{E}_Y\left[\left|h^2_s-\mathbb{E}_Y\left[h^2_s\right]\right|^2\right]
			\right)
			\mathrm{Lip}\left(\tilde{\indicator}_k\right)
			\left|\mathbb{E}_Y\left[h^1_s-h^2_s\right]\right|^2
			\\&\leq
			2\left(1+\left((k+1)^2+\mathrm{Lip}(\observ) \boundB e^{\oneSidedLip t}\right)\mathrm{Lip}\left(\tilde{\indicator}_k\right)\right)
			\mathbb{E}_Y\left[\left|h^1_s-h^2_s\right|^2\right],
		\end{align*}
		where we used that $\tilde{\indicator}_k\leq\indicator_{[0,k+1]}$ and the variance bound \eqref{variance bound potential fixed point} to derive the last inequality.\newline
		
		The variance bounds \eqref{variance bound tilde u} and \eqref{variance bound potential fixed point} also imply that 
		\begin{align}\label{ineqaulity state-observation covariance}
			\norm{\mathbb{Cov}_Y\left[\tilde{u}^2_s,h^2_s\right]}{L(R^{d_x},\HHilbert)}\leq \sqrt{\mathrm{Lip}(\observ)} \boundB e^{\oneSidedLip t}.
		\end{align}
		
		If we now take the supremum on the time interval $[0,T]$ and the conditional expectation $\mathbb{E}_Y$, standard Cauchy--Schwarz inequalities, together with \eqref{proof obs continuity - inequ Ito correction}, \eqref{proof obs continuity - inequ obsIto} and the one-sided Lipschitz condition \eqref{one sided Lipschitz}, we derive that there exists a constant $\kappa_1(T)$, that only depends on the timeframe $T$, such that
		\begin{align}\label{Estimate u1-u2}
			\begin{split}
				&\mathbb{E}_Y\left[\sup_{t\leq T}\norm{\tilde{u}^1_t-\tilde{u}^2_t}{\HHilbert}^2\right]
				\\&\leq
				\kappa_1(T)
				\int_{0}^{T} \mathbb{E}_Y\left[\norm{\tilde{u}^1_s-\tilde{u}^2_s}{\HHilbert}^2\right]
				+
				\mathbb{E}_Y\left[\norm{h^1_s-h^2_s}{\HHilbert}^2\right]
				\mathrm{d}s
				\\&\phantom{=}+
				\kappa_1(T)\int_{0}^{t}
				\left(\mathbb{E}_Y\left[\norm{\tilde{u}^1_s-\tilde{u}^2_s}{\HHilbert}^2\right]
				+
				\mathbb{E}_Y\left[\norm{h^1_s-h^2_s}{\HHilbert}^2\right]\right)
				\tilde{\indicator}_{k}\left(\left|\mathbb{E}_Y\left[h^1_s\right]\right|\right)
				\mathbb{E}_Y\left[\norm{h^1_s+\mathbb{E}_Y\left[h^1_s\right]}{\HHilbert}^2\right]
				\mathrm{d}s
				\\&\phantom{=}+
				2\mathbb{E}_Y\left[\sup_{t\leq T}
				\left|\int_{0}^{t}
				\left\langle \tilde{u}^1_s-\tilde{u}^2_s,
				\left(\mathbb{Cov}_Y\left[\tilde{u}^1_s,h^1_s\right]-\mathbb{Cov}_Y\left[\tilde{u}^2_s,h^2_s\right]\right) R^{-1}_s\mathrm{d}Y_s\right\rangle_{\HHilbert}\right|
				\right]
				\\&\phantom{=}+
				2
				\mathbb{E}_Y\left[
				\sup_{t\leq T}
				\left|
				\int_{0}^{t}
				\left\langle \tilde{u}^1_s-\tilde{u}^2_s,\left(B(\tilde{u}^1_s)-B(\tilde{u}^2_s)\right)\mathrm{d}W_s
				\right\rangle_{\HHilbert}
				\right|
				\right].
			\end{split}
		\end{align}

		Note that due to \eqref{variance bound potential fixed point} we get
		\begin{align*}
			\tilde{\indicator}_{k}\left(\left|\mathbb{E}_Y\left[h^1_s\right]\right|\right)
			\mathbb{E}_Y\left[\norm{h^1_s+\mathbb{E}_Y\left[h^1_s\right]}{\HHilbert}^2\right]
			&\leq
			2
			\mathbb{E}_Y\left[\norm{h^1_s-\mathbb{E}_Y\left[h^1_s\right]}{\HHilbert}^2\right]
			+
			8
			\tilde{\indicator}_{k}\left(\left|\mathbb{E}_Y\left[h^1_s\right]\right|\right)
			\norm{\mathbb{E}_Y\left[h^1_s\right]}{\HHilbert}^2
			\\&\leq
			2\mathrm{Lip}(\observ) \boundB e^{\oneSidedLip t}
			+
			8(k+1)^2
		\end{align*}
		Thus if we now use the specific form of the observations \eqref{specific form of observations} and take the full expectation in \eqref{Estimate u1-u2} we derive
		\begin{align*}
			&\mathbb{E}\left[\sup_{t\leq T}\norm{\tilde{u}^1_t-\tilde{u}^2_t}{\HHilbert}^2\right]
			\\&\leq
			\kappa_2(T,k)
			\int_{0}^{T} \mathbb{E}\left[\norm{\tilde{u}^1_s-\tilde{u}^2_s}{\HHilbert}^2\right]
			+
			\mathbb{E}\left[\norm{h^1_s-h^2_s}{\HHilbert}^2\right]
			\mathrm{d}s
			\\&+
			2\mathbb{E}\left[\int_{0}^{T}
			\mathbb{E}_Y\left[
			\left|\left\langle \tilde{u}^1_s-\tilde{u}^2_s,
			\left(\mathbb{Cov}_Y\left[\tilde{u}^1_s,h^1_s\right]-\mathbb{Cov}_Y\left[\tilde{u}^2_s,h^2_s\right]\right) R^{-1}_s 
			\tilde{\indicator}_{l}\left(\mathbb{E}_Y\left[\left|\observ(u^{\mathrm{ref}}_s)\right|^2\right]\right)
			\observ(u^{\mathrm{ref}}_s)\right\rangle_{\HHilbert}\right|
			\right]\right]
			\\&+
			2\mathbb{E}\left[\sup_{t\leq T}\int_{0}^{t}
			\left\langle \tilde{u}^1_s-\tilde{u}^2_s,
			\left(\mathbb{Cov}_Y\left[\tilde{u}^1_s,h^1_s\right]-\mathbb{Cov}_Y\left[\tilde{u}^2_s,h^2_s\right]\right) R^{-1}_s \tilde{\indicator}_{l}\left(\mathbb{E}_Y\left[\left|\observ(u^{\mathrm{ref}}_s)\right|^2\right]\right) \Gamma_s\mathrm{d}V_s \right\rangle_{\HHilbert}
			\right]
			\\&+
			2\mathbb{E}\left[\sup_{t\leq T}\int_{0}^{t}
			\left\langle \tilde{u}^1_s-\tilde{u}^2_s,\left(\Bvol(\tilde{u}^1_s)-\Bvol(\tilde{u}^2_s)\right)\mathrm{d}W_s
			\right\rangle_{\HHilbert}
			\right]
			,
		\end{align*}
		for some constant $\kappa_2(T,k)$, where we of course used that $\mathbb{E}\left[~\mathbb{E}_{Y}\left[\cdot\right]\right]=\mathbb{E}\left[\cdot\right]$.\newline
		
		To dominate the second term on the right hand side of the inequality  we use \eqref{covariance difference} together with the fact that $\tilde{\indicator}_{k}\left(\mathbb{E}_Y\left[\left|\observ(u^{\mathrm{ref}}_s)\right|^2\right]\right)\mathbb{E}_Y\left[
		\left|\observ(u^{\mathrm{ref}}_s)\right|^2\right] \leq (k+1)$. For the other two terms we use the Burkholder--Davis--Gundy inequality together with \eqref{proof obs continuity - inequ Ito correction} and \eqref{proof obs continuity - inequ obsIto} to derive that there exists a constant $\kappa_3\left(T,k,l\right)>0$ we have
		\begin{align*}
			&\mathbb{E}\left[\sup_{t\leq T}\norm{\tilde{u}^1_t-\tilde{u}^2_t}{\HHilbert}^2\right]
			\leq
			\kappa_3\left(T,k,l\right)
			\int_{0}^{T} \mathbb{E}\left[\norm{\tilde{u}^1_s-\tilde{u}^2_s}{\HHilbert}^2\right]
			+
			\mathbb{E}\left[\norm{h^1_s-h^2_s}{\HHilbert}^2\right]
			\mathrm{d}s
			,
		\end{align*}
		which by the (deterministic) Grönwall Lemma implies
		\begin{align*}
			\mathbb{E}\left[\sup_{t\leq T}\norm{\tilde{u}^1_t-\tilde{u}^2_t}{\HHilbert}^2\right]
			\leq
			\kappa_3\left(T,k,l\right)
			\exp\left(T~\kappa_3\left(T,\norm{\observ}{\infty}\right)\right)
			\int_{0}^{T} 
			\mathbb{E}\left[\norm{h^1_s-h^2_s}{\HHilbert}^2\right]
			\mathrm{d}s,
		\end{align*}
		and thus for $T$ small enough we indeed have the desired contraction property.\newline
		
		\uline{Step 2:} \textit{The stopping argument.}\newline
		
		First we define the stopping times which we use for our argument by 
		\begin{align}\label{definition stopping times}
			\begin{split}
				\tau^k
				&:=\inf\left\{~t\geq 0~:~\left|\mathbb{E}_Y\left[\observ(u^k_t)\right]\right|^2> k ~\right\}
				\\
				\tau^l_{\mathrm{ref}}
				&:=
				\inf\left\{~t\geq 0~:~\mathbb{E}_Y\left[\left|\observ(u^{\mathrm{ref}}_t)\right|^2\right]> l  ~\right\},
			\end{split}
		\end{align}
		and note that both are stopping times with respect to the filtration generated by $Y$. This implies that for any stochastic process $(z_t)_{t\geq 0}$ and any (suitably integrable) functions $f,g$, the identities
		\begin{align*}
			g\left(\mathbb{E}_Y\left[f(z_{\min\{\tau^k,t\}})\right]\right)
			&=
			\left. g\left(\mathbb{E}_Y\left[f(z_{s})\right]\right)
			\right|_{s=\min\{\tau^k,t\}}
			\\
			g\left(\mathbb{E}_Y\left[f(z_{\min\{\tau^l_{\mathrm{ref}},t\}})\right]\right)
			&=
			\left. g\left(\mathbb{E}_Y\left[f(z_{s})\right]\right)
			\right|_{s=\min\{\tau^l_{\mathrm{ref}},t\}}
		\end{align*}
		hold and therefore $\bar{u}^k$ is a solution to the EnKBF \eqref{nonlinear EnKBF} on the random time interval $[0,\min\{\tau^k,\tau^l_{\mathrm{ref}}\}]$. By the uniqueness of \eqref{fixed point equation stopped dynamics}, $\bar{u}^k$ and $\bar{u}^{k+1}$ must even coincide on $[0,\min\{\tau^k,\tau^l_{\mathrm{ref}}\}]$.  Thus we can construct a solution to \eqref{nonlinear EnKBF} using the solutions to \eqref{fixed point equation stopped dynamics}. In order to conclude existence and uniqueness of the EnKBF, we just have to show that
		\begin{align*}
			\bigcup_{k,l\in\N}\left\{\tau^k> T\right\}\cap\left\{\tau^l_{\mathrm{ref}}> T\right\}
		\end{align*}
		defines a covering of the sample space almost surely.\newline
		
		To this end we first note that
		\begin{align*}
			\mathrm{d}\norm{\bar{u}^k_t}{\HHilbert}^2
			&=
			2\prescript{}{\VHilbert'}{\left\langle \Adrift(\bar{u}^k_t), \bar{u}^k_t
				\right\rangle_{\VHilbert}}\mathrm{d}t
			+
			\left\langle\bar{u}^k_t,\Bvol(\bar{u}^k_t)\mathrm{d}\bar{W}_t\right\rangle_{\HHilbert}
			+
			\mathrm{tr}_{\HHilbert}\left[\Bvol(\bar{u}^k_t)\sqrt{\Qcor}
			\left(\Bvol(\bar{u}^k_t)\sqrt{\Qcor}\right)'\right]\mathrm{d}t
			\\&\phantom{=}+
			2\left\langle\bar{u}^k_t,
			\mathbb{Cov}_Y\left[\bar{u}^k_t,\observ^{k}\left(\bar{u}^k_t\right)\right] R^{-1}_t
			\left(\mathrm{d}Y^{l}_t-
			\frac{\observ^{k}\left(\bar{u}^k_t\right)+\mathbb{E}_Y\left[\observ^{k}\left(\bar{u}^k_t\right)\right]}{2}
			\mathrm{d}t
			\right)			
			\right\rangle_{\HHilbert}
			\\&\phantom{=}+
			\mathrm{tr}_{\HHilbert}\left[\mathbb{Cov}_Y\left[\bar{u}^k_t,\observ^{k}\left(\bar{u}^k_t\right)\right] R^{-1}_t \mathbb{Cov}_Y\left[\observ^{k}\left(\bar{u}^k_t\right),\bar{u}^k_t\right]\right].
		\end{align*}
		
		Taking the conditional expectation thus gives us
		\begin{align}\label{second moment inequality 1}
			\begin{split}
				&\mathrm{d}\mathbb{E}_Y\left[\norm{\bar{u}^k_t}{\HHilbert}^2\right]
				\\&=
				2\mathbb{E}_Y\left[\prescript{}{\VHilbert'}{\left\langle \Adrift(\bar{u}^k_t), \bar{u}^k_t
					\right\rangle_{\VHilbert}}\right]
				\mathrm{d}t
				+
				\mathbb{E}_Y\left[\mathrm{tr}_{\HHilbert}\left[\Bvol(\bar{u}^k_t)\sqrt{\Qcor}
				\left(\Bvol(\bar{u}^k_t)\sqrt{\Qcor}\right)'\right]\right]\mathrm{d}t
				\\&\phantom{=}+
				2\mathbb{E}_Y\left[
				\left\langle\bar{u}^k_t,
				\mathbb{Cov}_Y\left[\bar{u}^k_t,\observ^{k}\left(\bar{u}^k_t\right)\right] R^{-1}_t
				\left(\mathrm{d}Y^{l}_t-
				\frac{\observ^{k}\left(\bar{u}^k_t\right)+\mathbb{E}_Y\left[\observ^{k}\left(\bar{u}^k_t\right)\right]}{2}
				\mathrm{d}t
				\right)			
				\right\rangle_{\HHilbert}
				\right]
				\\&\phantom{=}+
				\tilde{\indicator}_{l}\left(\mathbb{E}_Y\left[\left|\observ(u^{\mathrm{ref}}_t)\right|^2\right]\right)^2
				\mathrm{tr}_{\HHilbert}\left[\mathbb{Cov}_Y\left[\bar{u}^k_t,\observ^{k}\left(\bar{u}^k_t\right)\right] R^{-1}_t \mathbb{Cov}_Y\left[\observ^{k}\left(\bar{u}^k_t\right),\bar{u}^k_t\right]\right]\mathrm{d}t.
			\end{split}
		\end{align}
		
		The first two terms on the right hand side can be bounded using the growth condition \eqref{growth condition} and the diffusivity bound \eqref{bounded signal diffusion}. The last term is just the squared shadow 2-norm of the operator $\mathbb{Cov}_Y\left[\bar{u}^k_t,\observ^{k}\left(\bar{u}^k_t\right)\right] R^{-1/2}_t$, which we can estimate using Parseval and the robust variance bound \eqref{variance bound tilde u} as
		\begin{align*}
			&\mathrm{tr}_{\HHilbert}\left[\mathbb{Cov}_Y\left[\bar{u}^k_t,\observ^{k}\left(\bar{u}^k_t\right)\right] R^{-1}_t \mathbb{Cov}_Y\left[\observ^{k}\left(\bar{u}^k_t\right),\bar{u}^k_t\right]\right]
			\leq
			\mathrm{tr}_{\HHilbert}\left[\mathbb{Cov}_Y\left[\bar{u}^k_t,\observ\left(\bar{u}^k_t\right)\right] R^{-1}_t \mathbb{Cov}_Y\left[\observ\left(\bar{u}^k_t\right),\bar{u}^k_t\right]\right]
			\\&\leq
			\sum_{j\in\N} \mathbb{Cov}_Y\left[\left\langle\nu_j,\bar{u}^k_t\right\rangle_{\HHilbert},\observ\left(\bar{u}^k_t\right)\right]R^{-1}_t
			\mathbb{Cov}_Y\left[\observ\left(\bar{u}^k_t\right),\left\langle\nu_j,\bar{u}^k_t\right\rangle_{\HHilbert}\right]
			\\&\leq
			\sum_{j\in\N}
			\mathbb{E}_Y\left[\left\langle\nu_j,
			\bar{u}^k_t-\mathbb{E}_Y\left[\bar{u}^k_t\right]
			\right\rangle_{\HHilbert}^2\right] \left|R^{-1}_t\right|
			\mathbb{E}_Y\left[\left|
			\observ(\bar{u}^k_t)-\mathbb{E}_Y\left[\observ(\bar{u}^k_t)\right]
			\right|^2\right]
			\\&\leq
			\mathbb{E}_Y\left[\norm{\bar{u}^k_t-\mathbb{E}_Y\left[\bar{u}^k_t\right]}{\HHilbert}^2\right]
			\left|R^{-1}_t\right|
			\mathbb{E}_Y\left[\left|
			\observ(\bar{u}^k_t)-\mathbb{E}_Y\left[\observ(\bar{u}^k_t)\right]
			\right|^2\right]^2
			\leq \mathrm{Lip}(\observ)\left|R^{-1}_t\right| \boundB^2 e^{2\oneSidedLip t}.
		\end{align*}
		
		Thus we can bound \eqref{second moment inequality 1} by
		\begin{align}\label{second moment inequality 2}
			\begin{split}
				\mathrm{d}\mathbb{E}_Y\left[\norm{\bar{u}^k_t}{\HHilbert}^2\right]
				&=
				\left(2\coerB\mathbb{E}_Y\left[\norm{\bar{u}^k_t}{\HHilbert}^2\right] + 2\coerC
				+ \boundB
				+
				\mathrm{Lip}(\observ)\left|R^{-1}_t\right| \boundB^2 e^{2\oneSidedLip t}
				\right)\mathrm{d}t
				\\&\phantom{=}+2
				\mathbb{E}_Y\left[
				\left\langle\bar{u}^k_t,
				\mathbb{Cov}_Y\left[\bar{u}^k_t,\observ^{k}\left(\bar{u}^k_t\right)\right]
				\right]
				R^{-1}_t
				\mathrm{d}Y^{l}_t\right\rangle_{\HHilbert}
				\\&\phantom{=}-2
				\mathbb{E}_Y\left[
				\left\langle\bar{u}^k_t,
				\mathbb{Cov}_Y\left[\bar{u}^k_t,\observ^{k}\left(\bar{u}^k_t\right)\right] R_t^{-1} \frac{\observ^k(\bar{u}^k_t)+\mathbb{E}_Y\left[\observ^k(\bar{u}^k_t)\right]}{2}
				\right\rangle_{\HHilbert}
				\right]\mathrm{d}t
				.
			\end{split}
		\end{align}
		
		We use \eqref{ineqaulity state-observation covariance} and $\tilde{\indicator}_k\leq 1$ to derive that
		\begin{align*}
			&
			\left|
			\mathbb{E}_Y\left[
			\left\langle\bar{u}^k_t,
			\mathbb{Cov}_Y\left[\bar{u}^k_t,\observ^{k}\left(\bar{u}^k_t\right)\right] R_t^{-1} \frac{\observ^k(\bar{u}^k_t)+\mathbb{E}_Y\left[\observ^k(\bar{u}^k_t)\right]}{2}
			\right\rangle_{\HHilbert}
			\right]
			\right|
			\\& \leq
			2\mathbb{E}_Y\left[\norm{\bar{u}^k_t}{\HHilbert}^2\right]
			+
			\tilde{\indicator}_{k}\left(\left|\mathbb{E}_Y\left[\observ(\bar{u}^k_t)\right]\right|\right)^2
			\norm{\mathbb{Cov}_Y\left[\bar{u}^k_t,\observ^{k}\left(\bar{u}^k_t\right)\right]}{\LinSpace{R^{d_x}}{\HHilbert}}^2
			\left|R^{-1}_t\right|
			\mathbb{E}_Y\left[|\observ(\bar{u}^k_t)|^2\right]
			\\& \leq
			\left(
			2
			+
			\mathrm{Lip}(\observ)^3
			\boundB^2 e^{2\oneSidedLip t} \left|R^{-1}_t\right|		
			\right)
			\mathbb{E}_Y\left[\norm{\bar{u}^k_t}{\HHilbert}^2\right]
			+
			\mathrm{Lip}(\observ)
			\boundB^2 e^{2\oneSidedLip t} \left|R^{-1}_t\right|
			\left|\observ(0)\right|^2.
		\end{align*}
		
		Thus, we note that there exist constants $\kappa_4(T)$ and $\kappa_5(T)$, only depending on the timeframe $T$, such that
		\begin{align*}
			\begin{split}
				\mathrm{d}\mathbb{E}_Y\left[\norm{\bar{u}^k_t}{\HHilbert}^2\right]
				&=
				\left(\kappa_4(T) \mathbb{E}_Y\left[\norm{\bar{u}^k_t}{\HHilbert}^2\right]
				+
				\kappa_5(T)
				\right)\mathrm{d}t
				\\&\phantom{=}+
				2
				\mathbb{E}_Y\left[
				\left\langle\bar{u}^k_t,
				\mathbb{Cov}_Y\left[\bar{u}^k_t,\observ^{k}\left(\bar{u}^k_t\right)\right]
				\right]
				R^{-1}_t
				\mathrm{d}Y^{l}_t\right\rangle_{\HHilbert}
				,
			\end{split}
		\end{align*}
		which by using the explicit form of the observations $Y$ can be rewritten as
		\begin{align}\label{second moment inequality 2}
			\begin{split}
				\mathrm{d}\mathbb{E}_Y\left[\norm{\bar{u}^k_t}{\HHilbert}^2\right]
				&=
				\left(\kappa_4(T) \mathbb{E}_Y\left[\norm{\bar{u}^k_t}{\HHilbert}^2\right]
				+
				\kappa_5(T)
				\right)\mathrm{d}t
				\\&\phantom{=}+
				2\mathbb{E}_Y\left[
				\left\langle\bar{u}^k_t,
				\mathbb{Cov}_Y\left[\bar{u}^k_t,\observ^{k}\left(\bar{u}^k_t\right)\right]
				\right]
				R^{-1}_t \tilde{\indicator}_{l}\left(\mathbb{E}_Y\left[\left|\observ(u^{\mathrm{ref}}_t)\right|^2\right]\right)
				\observ(u^{\mathrm{ref}}_t)
				\right\rangle_{\HHilbert}\mathrm{d}t
				\\&\phantom{=}+
				2\mathbb{E}_Y\left[
				\left\langle\bar{u}^k_t,
				\mathbb{Cov}_Y\left[\bar{u}^k_t,\observ^{k}\left(\bar{u}^k_t\right)\right]
				\right]
				R^{-1}_t \tilde{\indicator}_{l}\left(\mathbb{E}_Y\left[\left|\observ(u^{\mathrm{ref}}_t)\right|^2\right]\right)
				\Gamma_t\mathrm{d}V_t
				\right\rangle_{\HHilbert}.
			\end{split}
		\end{align}
		
		Again using \eqref{ineqaulity state-observation covariance} and the Burkholder--Davis--Gundy inequality we derive that there exist constants $\kappa_6(T)$, depending solely on $T$, and $\kappa_7(T,l)$, depending on $T$ and $l$, such that 
		\begin{align}\label{second moment inequality 2}
			\begin{split}
				\mathbb{E}\left[\sup_{t\leq T}\mathbb{E}_Y\left[\norm{\bar{u}^k_t}{\HHilbert}^2\right]\right]
				&=
				\kappa_6(T)
				\int_{0}^{T}
				\mathbb{E}_Y\left[\norm{\bar{u}^k_t}{\HHilbert}^2\right]\mathrm{d}t
				+
				\mathbb{E}_Y\left[\norm{u_0}{\HHilbert}^2\right]+T
				\kappa_7(T,l).
			\end{split}
		\end{align}
		
		Thus, by the Grönwall Lemma, we derive that for fixed $l$ and $T$
		\begin{align}\label{bound 2nd absolute moment}
			\mathbb{E}\left[\sup_{t\leq T}\mathbb{E}_Y\left[\norm{\bar{u}^k_t}{\HHilbert}^2\right]\right]
			\leq
			\exp\left(T\kappa_6(T)\right) \left(\mathbb{E}_Y\left[\norm{u_0}{\HHilbert}^2\right]+T\kappa_7(T,l)\right),
		\end{align}
		which, implies that almost surely there exists a $k$ such that $\sup_{t\leq T}\mathbb{E}_Y\left[\norm{\bar{u}^k_t}{\HHilbert}^2\right]\leq k$. By the Lipschitz continuity of $\observ$ and the inequality $\left|\mathbb{E}_Y\left[\observ(\bar{u}^k_t)\right]\right|
		\leq |\observ(0)|+\mathrm{Lip}(\observ)
		\sqrt{\mathbb{E}_Y\left[\norm{\bar{u}^k_t}{\HHilbert}^2\right]}$ this then in turn implies that
		\begin{align*}
			\bigcup_{k\in\N}\left\{\tau^k>T\right\}\cap\left\{\tau^l_{\mathrm{ref}}>T\right\}
			=
			\left\{\tau^l_{\mathrm{ref}}>T\right\}~\text{almost surely.}
		\end{align*}
		
		Since $\sup_{t\leq T}\mathbb{E}_Y\left[\left|\observ(u^{\mathrm{ref}}_t)\right|^2\right]$ is finite almost surely, we can thus indeed conclude that there exists a solution to the EnKBF \eqref{nonlinear EnKBF}, defined on every event $\left\{\tau^k>T\right\}\cap\left\{\tau^l_{\mathrm{ref}}>T\right\}$ by the sequence of $\bar{u}^k$.\newline
		
		Uniqueness follows from uniqueness of the stopped dynamics \eqref{partially stopped mf EnKBF}.
		
	\end{proof}
	
	\begin{Rmk}
		Note that even though in the proof above we used the specific form of the observations $Y$, it actually does  not  matter that the true observation function and the modelled observation function coincide, i.e. if $\mathrm{d}Y_t=\mathfrak{C}(\bar{X}_t)\mathrm{d}t+\Gamma_t\mathrm{d}V_t$ with $\mathfrak{C}\neq \observ$, then the proof would still hold, as long as $\mathfrak{C}$ is assumed to be Lipschitz. Therefore, as an immediate corollary of our chosen fixed point argument, one derives the continuity of the EnKBF with respect to perturbations of the modelled observations $H$. The continuous dependence on the signal parameters $\Adrift$, $\Bvol$ and the initial condition $u_0$ can be shown as well. Only the robustness with respect to the observation stream $Y$ is a delicate matter due to the discontinuity of the Itô--Lyons map.
	\end{Rmk}

	\section{Quantitative propagation of chaos}\label{section prop of chaos}
	
	Next we show propagation of chaos, i.e. that the system of interacting SPDEs \eqref{nonlinear EnKBF - particle approximation} indeed converges (in an appropriate sense) to the McKean--Vlasov SPDE \eqref{nonlinear EnKBF}. For this we use a standard synchronous coupling approach, i.e. we compare \eqref{nonlinear EnKBF - particle approximation}  to a tensorized version of \eqref{nonlinear EnKBF} defined on the same probability space. 
	To this end we define conditionally\footnote{Conditioned on $Y$.} independent copies $\bar{u}^i,~i\in\N$ of the mean field process \eqref{nonlinear EnKBF} to be the solutions of
	\begin{align*}
		\mathrm{d}\bar{u}^i_t
		&=\Adrift(\bar{u}^i_t)\mathrm{d}t+\Bvol(\bar{u}^i_t)\mathrm{d}\bar{W}^i_t
		\\&\phantom{=}+
		\mathbb{Cov}_Y\left[\bar{u}^i_t,\observ(\bar{u}^i_t)\right] R^{-1}_t\left(\mathrm{d}Y_t-\frac{\observ(\bar{u}^i_t)+\mathbb{E}_Y\left[\observ(\bar{u}^i_t)\right]}{2}
		\mathrm{d}t
		\right),~i=1,\cdots,N,
	\end{align*}
	where $\bar{W}^i,~i\in\N$ are the same Wiener processes that also drive the particle system \eqref{nonlinear EnKBF - particle approximation}.  Furthermore
	we set $\bar{\mathfrak{u}}^N:=\left(\bar{u}^1,\cdots,\bar{u}^N\right)\in \HHilbert^{N}$ and make the following definition.
	
	\begin{Def}\label{Definition stopping time}
		We define the empirical observed accuracy
		\begin{align*}
			\mathcal{R}^N_\observ(\mathfrak{u}^N_s)
			:=
			\frac{1}{N}\sum_{i=1}^{N}\norm{\observ(u^{\mathrm{ref}}_{s})
				-
				\frac{
					\observ(u^i_s)+\mathbb{E}^N_\observ\left[u_s\right]
				}{2}}{\HHilbert}^2.
		\end{align*}
		
		We also define the corresponding hitting times for any $k\in\N$
		\begin{align*}
			\tau^{k}_{\sigma}
			&:=
			\inf\left\{~t\geq0~:~\sigma^{N}[\mathfrak{u}^N_t]>k~\right\}
			,~
			\tau^{k}_{\bar{\sigma}}:=
			\inf\left\{~t\geq0~:~\sigma^{N,\observ}[\bar{\mathfrak{u}}^N_t]>k~\right\}
			,\\
			\tau^{k}_{\mathcal{R}}
			&:=
			\inf\left\{~t\geq0~:~\mathcal{R}^N_{\observ}(\mathfrak{u}^N_t)>k~\right\}.
		\end{align*}
	\end{Def}

	\begin{Def}\label{Definition LLN}
		Furthermore we define the error of the law of large numbers by
		\begin{align*}
			\mathrm{LLN}^N_\observ(T)
			:=
			\int_{0}^{T}
			\norm{\mathbb{C}^{N}_{\observ}\left[\bar{\mathfrak{u}}^N_s\right]
				-
				\mathbb{Cov}_Y\left[\bar{u}_s,\observ(\bar{u}_s)\right]}{\HHilbert^{d_y} }^2
			+
			\left|\mathbb{E}^N_{\observ}\left[\bar{\mathfrak{u}}^N_s\right]-\mathbb{E}_Y\left[H(\bar{u}_s)\right]\right|^2
			\mathrm{d}s
			.
		\end{align*}
	\end{Def}
	
	Now we are able prove convergence of the particle system with implicit rates.\newline
	
	\begin{Thm}\label{prop of chaos - implicit rates}
		Assume that the conditions in Assumption \ref{signal assumptions} and \ref{bounded signal diffusion} are satisfied, and let $\tau^k:=\min\left\{\tau^{k}_{\sigma},\tau^{k}_{\bar{\sigma}},\tau^{k}_{\mathcal{R}}\right\}$. Then for any $p\in(0,1)$ there exists a constant $\kappa(T,k,p)$, such that
		\begin{align*}
			&\mathbb{E}\left[
			\sup_{t\leq \min\{T,\tau^k\}}
			\left(\frac{1}{N}\sum_{i=1}^{N}\norm{r^i_{\min\{t,\tau^k\}}}{\HHilbert}^2\right)^p
			\right]
			\leq
			\kappa(T,k,p)~
			\mathbb{E}\left[
			\left(
			\mathrm{LLN}^N_{\observ}(\min\{T,\tau^k\})
			\right)^p
			\right]
		\end{align*}
		
	\end{Thm}
	\begin{proof}
		We  note that since $u^i$ and $\bar{u}^i$ share the same initial conditions we have for any $t\geq 0,~i=1,\cdots,N$ that
		\begin{align*}
			r^i_t
			&=u^i_t-\bar{u}^i_t
			=
			\int_{0}^{t} \mathrm{d}\left(u^i_s-\bar{u}^i_s\right)
			\\&=
			\int_{0}^{t} \Adrift(u^i_s)-\Adrift(\bar{u}^i_s)~\mathrm{d}s
			+
			\int_{0}^{t} \Bvol(u^i_s)-\Bvol(\bar{u}^i_s)~\mathrm{d}\bar{W}^i_s
			\\&\phantom{=}+
			\int_{0}^{t} 
			\left(
			\mathbb{C}^{N}_{\observ}\left[\mathfrak{u}^N_s\right]
			-
			\mathbb{Cov}_Y\left[\bar{u}^i_s,\observ(\bar{u}^i_s)\right]
			\right) R^{-1}_s\left(\mathrm{d}Y_t-\frac{\observ(u^i_s)+\mathbb{E}^N_{\observ}\left[\mathfrak{u}^N_s\right]}{2}
			\mathrm{d}s
			\right)
			\\&\phantom{=}-
			\frac{1}{2}
			\int_{0}^{t}
			\mathbb{Cov}_Y\left[\bar{u}^i_s,\observ(\bar{u}^i_s)\right]
			R^{-1}_s
			\left(
			\observ(u^i_s)-\observ(\bar{u}^i_s) 
			+
			\mathbb{E}^N_{\observ}[\mathfrak{u}^N_s]-\mathbb{E}_Y\left[\observ(\bar{u}^i_s)\right] 
			\right)
			\mathrm{d}s.
		\end{align*}
		
		Therefore by using the concrete form of the observation process $\mathrm{d}Y_t=\observ(u^{\mathrm{ref}}_t)\mathrm{d}t+\Gamma_t\mathrm{d}V_t$ we derive from Itô's Lemma
		\begin{align*}
			&\norm{r^i_t}{\HHilbert}^2
			=2\int_{0}^{t}
			\prescript{}{\VHilbert'}{\left\langle
				\Adrift(u^i_s)-\Adrift(\bar{u}^i_s), u^i_s-\bar{u}^i_s
				\right\rangle_{\VHilbert}}
			\mathrm{d}s
			+
			2
			\int_{0}^{t}
			\left\langle
			u^i_s-\bar{u}^i_s, 
			\left(\Bvol(u^i_s)-\Bvol(\bar{u}^i_s)\right)~\mathrm{d}\bar{W}^i_s
			\right\rangle_{\HHilbert}
			\\&\phantom{=}+
			\int_{0}^{t}
			\norm{(\Bvol(u^i_s)-\Bvol(\bar{u}^i_s))\circ\sqrt{\Qcor}}{\HilbSchm{\UHilbert}{\HHilbert}}^2\mathrm{d}s
			\\&\phantom{=}+
			2\int_{0}^{t}
			\left\langle
			u^i_s-\bar{u}^i_s, 
			\left(
			\mathbb{C}^{N}_{\observ}\left[\mathfrak{u}^N_s\right]
			-
			\mathbb{Cov}_Y\left[\bar{u}^i_s,\observ(\bar{u}^i_s)\right]
			\right)
			R^{-1}_s
			\Gamma_s
			\mathrm{d}V_s
			\right\rangle_{\HHilbert}
			\\&\phantom{=}+
			2\int_{0}^{t}
			\left\langle
			u^i_s-\bar{u}^i_s, 
			\left(
			\mathbb{C}^{N}_{\observ}\left[\mathfrak{u}^N_s\right]
			-
			\mathbb{Cov}_Y\left[\bar{u}^i_s,\observ(\bar{u}^i_s)\right]
			\right)
			R^{-1}_s
			\left(
			\observ(u^{\mathrm{ref}}_{s})
			-
			\frac{
				\observ(u^i_s)+\mathbb{E}^N_{\observ}\left[\mathfrak{u}^N_s\right]
			}{2}
			\right)
			\right\rangle_{\HHilbert}
			\mathrm{d}s
			\\&\phantom{=}+
			\int_{0}^{t}
			\left|\left\langle
			u^i_s-\bar{u}^i_s, 
			\left(
			\mathbb{C}^{N}_{\observ}\left[\mathfrak{u}^N_s\right]
			-
			\mathbb{Cov}_Y\left[\bar{u}^i_s,\observ(\bar{u}^i_s)\right]
			\right)\right\rangle_{\HHilbert}
			R^{-1/2}_s\right|^2
			\mathrm{d}s
			\\&\phantom{=}-
			\int_{0}^{t}
			\left\langle
			u^i_s-\bar{u}^i_s ,
			\mathbb{Cov}_Y\left[\bar{u}^i_s,\observ(\bar{u}^i_s)\right]
			R^{-1}_s
			\left(
			\observ(u^i_s)-\observ(\bar{u}^i_s) 
			+
			\mathbb{E}^N_\observ[\mathfrak{u}^N_s]-\mathbb{E}_Y\left[\observ(\bar{u}^i_s)\right] 
			\right)
			\right\rangle_{\HHilbert}
			\mathrm{d}s.
		\end{align*}
		
		Thus by forming the average and using the Lipschitz assumptions \eqref{one sided Lipschitz}, as well as elementary Cauchy--Schwarz inequalities, we derive that there exists a constant $\kappa_1(T)>0$, only depending on time, such that
		\small\begin{align}\label{Propagation of chaos proof - first inequality}
			\begin{split}
				&\frac{1}{N}\sum_{i=1}^{N}\norm{r^i_t}{\HHilbert}^2
				\leq
				\kappa_1(T)
				\int_{0}^{t}\frac{1}{N}\sum_{i=1}^{N}\norm{r^i_s}{\HHilbert}^2~\mathrm{d}s
				+\mathfrak{lm}_t
				\\&+	
				\int_{0}^{t} \frac{1}{N} \sum_{i=1}^{N} \norm{\mathbb{C}^{N}_{\observ}\left[\mathfrak{u}^N_s\right]
					-
					\mathbb{Cov}_Y\left[\bar{u}^i_s,\observ(\bar{u}^i_s)\right]}{\HHilbert^{d_y}}^2~
				\left(
				\norm{\observ(u^{\mathrm{ref}}_{s})
					-
					\frac{
						\observ(u^i_s)+\mathbb{E}^N_{\observ}\left[\mathfrak{u}^N_s\right]
					}{2}}{\HHilbert}^2
				+
				2|R^{-1}_s|
				\right)
				\mathrm{d}s
				\\&+	
				2\int_{0}^{t}
				\frac{1}{N} \sum_{i=1}^{N}
				\norm{\mathbb{Cov}_Y\left[\bar{u}^i_s,\observ(\bar{u}^i_s)\right]}{\HHilbert^{d_y}}^2 
				\left(\left|\observ(u^i_s)-\observ(\bar{u}^i_s)\right|^2
				+
				\left|\mathbb{E}^N_{\observ}[\mathfrak{u}^N_s]-\mathbb{E}_Y\left[\observ(\bar{u}^i_s)\right]\right|^2
				\right)
				\mathrm{d}s,
			\end{split}
		\end{align}
		where
		\begin{align*}
			\mathrm{d}\mathfrak{lm}_t
			&:=
			\frac{2}{N} \sum_{i=1}^{N}
			\left\langle
			u^i_s-\bar{u}^i_s, 
			\left(\Bvol(u^i_s)-\Bvol(\bar{u}^i_s)\right)~\mathrm{d}\bar{W}^i_s
			\right\rangle_{\HHilbert}
			\\&\phantom{=}	+
			\frac{2}{N} \sum_{i=1}^{N}
			\left\langle
			u^i_s-\bar{u}^i_s, 
			\left(
			\mathbb{C}^{N}_{\observ}\left[\mathfrak{u}^N_s\right]
			-
			\mathbb{Cov}_Y\left[\bar{u}^i_s,\observ(\bar{u}^i_s)\right]
			\right)
			R^{-1}_s
			\Gamma_s
			\mathrm{d}V_s
			\right\rangle_{\HHilbert}
		\end{align*}
		is a local martingale. Its concrete form will not matter to our further calculations as we intend to use the stochastic Grönwall Lemma \cite{Scheutzow}.\newline
		
		First we note that the (conditional) covariance operator $\mathbb{Cov}_Y\left[\bar{u}^i_s,\observ(\bar{u}^i_s)\right]=\mathbb{Cov}_Y\left[\bar{u}_s,\observ(\bar{u}_s)\right]$ is independent of $i=1,\cdots,N$ and that it can be uniformly bounded on any finite time interval $[0,T]$ due to the variance bound \eqref{variance bound tilde u}, which helps us ignore the quadratic covariation of $\mathfrak{lm}$ and is thus especially suited for the one-sided Lipschitz conditions we encounter.\newline
		
		Next we note that
		\begin{align*}
			\left|
			\mathbb{E}^N_{\observ}[\mathfrak{u}^N_s]-\mathbb{E}_Y\left[\observ(\bar{u}^i_s)\right]
			\right|^2
			\leq
			2\mathrm{Lip}(\observ)^2 \frac{1}{N}\sum_{i=1}^{N}\norm{r^i_s}{\HHilbert}^2
			+2\left|\mathbb{E}^N_{\observ}\left[\bar{\mathfrak{u}}^N_s\right]-\mathbb{E}_Y\left[\observ\left(\bar{u}_s\right)\right]\right|^2.
		\end{align*}
		
		Finally we note that
		\begin{align*}
			&\norm{\mathbb{C}^{N}_{\observ}\left[\mathfrak{u}^N_s\right]
				-
				\mathbb{Cov}_Y\left[\bar{u}^i_s,\observ(\bar{u}^i_s)\right]}{\HHilbert^{d_y}}
			\\&\leq
			\norm{\mathbb{C}^{N}_{\observ}\left[\mathfrak{u}^N_s\right]
				-
				\mathbb{C}^{N}_{\observ}\left[\bar{\mathfrak{u}}^N_s\right]}{\HHilbert^{d_y}}
			+
			\norm{\mathbb{C}^{N}_{\observ}\left[\bar{\mathfrak{u}}^N_s\right]
				-
				\mathbb{Cov}_Y\left[\bar{u}^i_s,\observ(\bar{u}^i_s)\right]}{\HHilbert^{d_y}}
			\\&\leq
			\norm{\frac{1}{N}\sum_{i=1}^{N}
				(u^i_s-\mathbb{E}^N[\mathfrak{u}^N_s])\left(\observ(u^i_s)-\observ(\bar{u}^i_s)\right)'
			}{\HHilbert^{d_y}}
			+
			\norm{\frac{1}{N}\sum_{i=1}^{N}
				(u^i_s-\bar{u}^i_s)\left(\observ(\bar{u}^i_s)-\mathbb{E}^N_{\observ}[\bar{\mathfrak{u}}^N_s]\right)'
			}{\HHilbert^{d_y}}
			\\&\phantom{=}+
			\norm{\mathbb{C}^{N}_{\observ}\left[\bar{\mathfrak{u}}^N_s\right]
				-
				\mathbb{Cov}_Y\left[\bar{u}^i_s,\observ(\bar{u}^i_s)\right]}{\HHilbert^{d_y}}
			\\&\leq
			\left(\sqrt{\frac{\mathrm{Lip}(\observ)}{N}\sum_{i=1}^{N}\norm{u^i_s-\mathbb{E}^N[\mathfrak{u}^N_s]}{\HHilbert}^2}
			+
			\sqrt{\frac{1}{N}\sum_{i=1}^{N}\left|\observ(\bar{u}^i_s)-\mathbb{E}^N_{\observ}[\bar{\mathfrak{u}}^N_s]\right|^2}
			\right)
			\sqrt{\frac{1}{N}\sum_{i=1}^{N}\norm{r^i_s}{\HHilbert}^2}
			\\&\phantom{=}+
			\norm{\mathbb{C}^{N}_{\observ}\left[\bar{\mathfrak{u}}^N_s\right]
				-
				\mathbb{Cov}_Y\left[\bar{u}^i_s,\observ(\bar{u}^i_s)\right]}{\HHilbert^{d_y}}
			.
		\end{align*}
		
		Using the notation of Definition \ref{Definition stopping time} this allows us to further estimate inequality \eqref{Propagation of chaos proof - first inequality}. Thus there exists a constant $\kappa_2(T)>0$ such that
		\begin{align}\label{Propagation of chaos proof - Gronwall inequality}
			\begin{split}
				\frac{1}{N}\sum_{i=1}^{N}\norm{r^i_t}{\HHilbert}^2
				&\leq
				\kappa_2(T)
				\int_{0}^{t}
				\left(
				1+\left(\sigma^N[\mathfrak{u}^N_s]+\sigma^{N,\observ}[\bar{\mathfrak{u}}^N_s]\right)
				\mathcal{R}^N_\observ(\mathfrak{u}^N_s)
				\right)
				\frac{1}{N}\sum_{i=1}^{N}\norm{r^i_s}{\HHilbert}^2~\mathrm{d}s
				\\&\phantom{=}+
				\kappa_2(T)\int_{0}^{t}
				\mathcal{R}^N_{\observ}(\mathfrak{u}^N_s)
				\norm{\mathbb{C}^{N}_{\observ}\left[\bar{\mathfrak{u}}^N_s\right]
					-
					\mathbb{Cov}_Y\left[\bar{u}^i_s,\observ(\bar{u}^i_s)\right]}{\HHilbert^{d_y}}^2\mathrm{d}s
				\\&\phantom{=}+
				\kappa_2(T)\int_{0}^{t}
				\left|\mathbb{E}^N_{\observ}\left[\bar{\mathfrak{u}}^N_s\right]-\mathbb{E}_Y\left[\observ(\bar{u}_s)\right]\right|^2
				\mathrm{d}s
				+\mathfrak{lm}_t
				.
			\end{split}
		\end{align}
		
		Using the stopping time $\tau^k$ we thus derive that there exists a constant $\kappa_3(T,k)$ only depending on timeframe $T$ and the stopping level $k$ such that
		\begin{align*}
			\frac{1}{N}\sum_{i=1}^{N}\norm{r^i_{\min\{t,\tau^k\}}}{\HHilbert}^2
			&\leq
			\kappa_3(T,k)
			\int_{0}^{\min\{t,\tau^k\}}
			\frac{1}{N}\sum_{i=1}^{N}\norm{r^i_s}{\HHilbert}^2
			\mathrm{d}s
			\\&\phantom{=}+
			\kappa_3(T,k)~
			\mathrm{LLN}^N_{\observ}(\min\{t,\tau^k\})
			+
			\mathfrak{lm}_{\min\{t,\tau^k\}}
			.
		\end{align*}
		
		Then by the stochastic Grönwall Lemma \cite[Theorem 4, equation (4)]{Scheutzow}, the claim of this lemma follows immediately.
		
	\end{proof}
	
	From this theorem one can also deduce the convergence in probability \cite{StannatLange}.\newline

	We say the convergence rates derived in Theorem \ref{prop of chaos - implicit rates} are implicit, as they require the processes to be stopped and the stopping times depend on the converging particle system itself. However on the stopped time intervals the rates are optimal in the sense that they correspond to the rates of convergence given by the law of large numbers. This is certainly far from the convergence result one would ultimately desire, but, for general signals and observation functions, nevertheless seems to be the current state of the art, even in the finite dimensional setting, where the same coupling method was used by \cite{StannatLange} to obtain similar results for Lipschitz signals and linear observation functions. For bounded signal and observation functions, as well as observation data $Y$ that is given by a Lipschitz continuous (w.r.t. time) rough path\footnote{Thus excluding Brownian observation noise that we treat here.}, \cite{Coghi et al} were able to prove explicit convergence rates. They  used a similar stopping argument as above, together with tail bounds for higher order empirical moments of the interacting ensemble. With this they were able to derive a logarithmic decay $\mathcal{O}\left(\log(N)^{-1}\right)$ of the error w.r.t. ensemble size $N$ without stopping, which is still far from the desired convergence rates corresponding to the law of large numbers.\newline
	
	Also  assuming the boundedness of the observation function $\observ$, we are able to prove the asymptotically, (almost) optimal convergence rate based on our exponential moment bounds in Proposition \ref{Lemma exponential moments} and the following additional assumption on the initial distribution.\newline
	
	\begin{Ass}\label{assumption exponential moments initial}
		We assume that for any $q>0$ there exists an $N_0(q)\in\N$ such that
		\begin{align*}
			\sup_{N\geq N_0(q)}
			\mathbb{E}\left[\exp\left(q ~\sigma^N[\mathfrak{u}^N_0]\right)\right]
			<+\infty.
		\end{align*}
	\end{Ass} 
	
	\begin{Rmk}
		This  assumption is always satisfied for deterministic initial conditions, as for $u_0\sim\delta_{v_0}$ for some $v_0\in \HHilbert$ one has $\sigma^N[\mathfrak{u}^N_0]=0$ for all $N\in\N$. For Gaussian initial conditions this relates to the domain of the moment generating function of $\chi^2$-distributions and for general random variables to large deviations of the empirical covariance matrix.
	\end{Rmk}
	
	We will not investigate further when this assumption is satisfied and just assume it holds. Then we are able to prove the following theorem.\newline
	
	\begin{Thm}\label{Thm optimal convergence}
		Besides the conditions in Assumption \ref{signal assumptions} and \ref{bounded signal diffusion}, we assume that $\observ$ is bounded and that Assumption \ref{assumption exponential moments initial} hold. Then for any $T<+\infty$, $p\in(0,1)$ and any $\nu\in(1,1/p)$ there exists an $\mathcal{N}_0\left(T,p,\nu\right)\in\N$ and a $\kappa\left(T,p,\nu\right)<+\infty$, such that for all $N\geq\mathcal{N}_0\left(T,p,\nu\right)$ we have
		\begin{align}\label{optimal rates before Hoelder}
			&\mathbb{E}\left[\sup_{t\leq T} \left(\frac{1}{N}\sum_{i=1}^{N}\norm{r^i_t}{\HHilbert}^2\right)^p\right]
			\leq
			\kappa\left(T,p,\nu\right)
			\mathbb{E}\left[\left(\mathrm{LLN}^N_{\observ}(T)\right)^{p\nu} \right]^{1/\nu}.
		\end{align}
		As a consequence we have for $\kappa(T,p):=\inf_{\nu\in(1,1/p)}\kappa(T,p,\nu)$
		\begin{align}\label{optimal rates - almost}
			\mathbb{E}\left[\sup_{t\leq T} \left(\frac{1}{N}\sum_{i=1}^{N}\norm{r^i_t}{\HHilbert}^2\right)^p\right]
			\leq
			\kappa\left(T,p\right)
			\mathbb{E}\left[\mathrm{LLN}^N_{\observ}(T) \right]^{p},
		\end{align}
		which in turn implies that for some constant $C\left(T,p,\norm{\observ}{\infty}\right)>0$ that
		\begin{align}\label{optimal rates - in N}
			\mathbb{E}\left[\sup_{t\leq T} \left(\frac{1}{N}\sum_{i=1}^{N}\norm{r^i_t}{\HHilbert}^2\right)^p\right]
			\leq
			C\left(T,p,\norm{\observ}{\infty}\right)
			N^{-p}.
		\end{align}
	\end{Thm}
	\begin{proof}
		First we note that inequality \eqref{Propagation of chaos proof - Gronwall inequality}, which was derived in the proof of Theorem \ref{prop of chaos - implicit rates}, can be further simplified when the observation function $\observ$ is assumed to be bounded, as then both $\sigma^{N,\observ}[\bar{\mathfrak{u}}^N_s]$ and  $\mathcal{R}^N_{\observ}(\mathfrak{u}^N_s)$ are uniformly bounded and thus there exists a constant $\kappa_4(T)$, such that
		\begin{align*}
			\begin{split}
				\frac{1}{N}\sum_{i=1}^{N}\norm{r^i_t}{\HHilbert}^2
				&\leq
				\kappa_4(T)
				\int_{0}^{t}
				\left(
				1+\sigma^N[\mathfrak{u}^N_s]
				\right)
				\frac{1}{N}\sum_{i=1}^{N}\norm{r^i_s}{\HHilbert}^2~\mathrm{d}s
				+
				\kappa_4(T)~\mathrm{LLN}^N_\observ(t)
				+\mathfrak{lm}_t
				.
			\end{split}
		\end{align*}
		
		The stochastic Grönwall inequality \cite{Scheutzow} thus tells us that for any $p\in(0,1)$ and  $\mu,\nu>1$ with $\frac{1}{\mu}+\frac{1}{\nu}=1$ and such that $p\nu<1$ we have
		\begin{align}
			\begin{split}
				&\mathbb{E}\left[\sup_{t\leq T} \left(\frac{1}{N}\sum_{i=1}^{N}\norm{r^i_t}{\HHilbert}^2\right)^p\right]
				\\&\leq
				\left( c_{p\nu}+1 \right)^{1/\nu}
				\mathbb{E}\left[
				\exp\left(p\mu~\kappa_4(T)
				\int_{0}^{T}
				\left(
				1+\sigma^N[\mathfrak{u}^N_s]
				\right)
				\mathrm{d}s
				\right)
				\right]
				~\kappa_4(T)~
				\mathbb{E}\left[	\left(\mathrm{LLN}^N_{\observ}(t)\right)^{p\nu} \right]^{1/\nu}
				,
			\end{split}
		\end{align}
		where 
		\begin{align*}
			c_{p\nu}
			:=
			\min\left\{4,1/(p\nu)\right\}~\frac{\pi p \nu}{\sin(\pi p \nu)}
			\xrightarrow{\nu\to 1/p}+\infty
			.
		\end{align*}
		
		Now we first note that $\mu=\frac{\nu}{\nu-1}$. Due to the exponential moment bounds in Proposition \ref{Lemma exponential moments} we have that for $q:=\frac{p\nu}{\nu-1}\kappa_4(T)T$
		\begin{align*}
			\mathbb{E}\left[
			\exp\left(p\mu\kappa_4(T)\int_{0}^{T}~\sigma^{N}[\mathfrak{u}^N_t]\right)\right]
			\leq
			(\pi+1) 
			\exp\left(\frac{q\left(e^{(2\oneSidedLip+1)T}-1\right)}{2(2\oneSidedLip+1)}\right)
			\mathbb{E}\left[\exp\left(2 q e^{(2\oneSidedLip+1)T} \sigma^N[\mathfrak{u}^N_0]\right)\right]
		\end{align*}
		and by Assumption \ref{assumption exponential moments initial} there exists an 
		\begin{align*}
			\mathcal{N}_0(T,p,\nu)
			:=N_0(q)
			=N_0\left(2 q e^{(2 \oneSidedLip+1)T)}\right)
			=N_0\left(2 \frac{p\nu}{\nu-1}\kappa_4(T)T e^{(2\oneSidedLip+1)T)}\right)
		\end{align*}
		such that 
		\begin{align*}
			\kappa_5(T,p\mu)
			:=
			\sup_{N\geq \mathcal{N}_0(T,p,\nu)}
			\mathbb{E}\left[
			\exp\left(2 \frac{p\nu}{\nu-1}\kappa_4(T)T e^{(2\oneSidedLip+1)T} \sigma^N[\mathfrak{u}^N_0]\right)
			\right]<+\infty
		\end{align*}
		and therefore for any $N\geq\mathcal{N}_0\left(T,p,\nu\right)$ we have
		\begin{align*}
			&\mathbb{E}\left[\sup_{t\leq T} \left(\frac{1}{N}\sum_{i=1}^{N}\norm{r^i_t}{\HHilbert}^2\right)^p\right]
			\leq
			\underbrace{\left( c_{p\nu}+1 \right)^{1/\nu}
				e^{\left(\frac{p\nu}{\nu-1} \kappa_4(T)T\right)}~
				\kappa_5\left(T,\frac{p\nu}{\nu-1}\right)
				\kappa_4(T)}_{:=\kappa_5\left(T,p,\nu\right)}
			\mathbb{E}\left[	\left(\mathrm{LLN}^N_{\observ}(t)\right)^{p\nu} \right]^{1/\nu},
		\end{align*}
		which proves our first claim. From this one can directly deduce \eqref{optimal rates - almost} via the Hölder inequality and using the fact that for $\nu\to 1$ both $\mathcal{N}(T,p,\nu)$ and $\kappa(T,p,\nu)$ blow up, i.e.
		\begin{align*}
			\mathcal{N}(T,p,\nu),\kappa(T,p,\nu)\xrightarrow{\nu\to 1}+\infty,
		\end{align*}
		as well  the blow up $\kappa(T,p,\nu)\xrightarrow{\nu\to 1/p}+\infty$, which in turn implies that the minimizer of $\kappa(T,p,\nu)$ for every fixed $T$ and $P$ must lie inside the interval $(1,1/p)$ and therefore \eqref{optimal rates - almost} is proven. Finally we are left to show \eqref{optimal rates - in N}. By Definition \ref{Definition LLN} the term $\mathrm{LLN}^N_\observ(T)$ consists of an error of the empirical mean and an error of the empirical covariance. First we estimate the  error of the empirical mean. Using the conditional independence of $\left(\bar{u}^i\right)_{i=1,\cdots,N}$ and the law of total variance we derive
		\begin{align}\label{proof convergence order in N - inequality 1}
			\begin{split}
				&\int_{0}^{T}
				\mathbb{E}\left[
				\left|\mathbb{E}^N_{\observ}\left[\bar{\mathfrak{u}}^N_s\right]-\mathbb{E}_Y\left[H(\bar{u}_s)\right]\right|^2
				\right]
				\mathrm{d}s
				=
				\mathbb{E}\left[\int_{0}^{T}
				\mathbb{E}_{Y}\left[
				\left|\frac{1}{N}\sum_{i=1}^{N}\observ\left(\bar{u}^i_s\right)-\mathbb{E}_Y\left[H(\bar{u}_s)\right]\right|^2
				\right]
				\mathrm{d}s\right]
				\\
				&=
				\mathbb{E}\left[\int_{0}^{T}
				\frac{\mathbb{Var}_{Y}\left[\observ\left(\bar{u}_s\right)\right]}{N}
				\mathrm{d}s\right]
				\leq
				\int_{0}^{T}
				\frac{\mathrm{Lip}\left(\observ\right)^2\boundB e^{\oneSidedLip s}}{N}
				\mathrm{d}s
				=\frac{\mathrm{Lip}\left(\observ\right)^2\boundB\left(e^{\oneSidedLip T}-1\right)}{\oneSidedLip~N}.
			\end{split}
		\end{align}
		
		Next we aim to dominate the error of the covariance. To this end we first note that
		\begin{align*}
			&\mathbb{E}_{Y}\left[	
			\norm{\mathbb{C}^{N}_{\observ}\left[\bar{\mathfrak{u}}^N_s\right]
				-
				\mathbb{Cov}_Y\left[\bar{u}_s,\observ(\bar{u}_s)\right]}{\HHilbert^{d_y} }^2	
			\right]
			\\&\leq
			2
			\mathbb{E}_{Y}\left[	
			\norm{
				\frac{1}{N}\sum_{i=1}^{N}\bar{u}^i_s \observ\left(\bar{u}^i_s\right)'
				-
				\mathbb{E}_{Y}\left[\bar{u}_s\observ\left(\bar{u}_s\right)\right]
			}{\HHilbert^{d_y} }^2	
			\right]
			\\&\phantom{\leq}+
			2
			\mathbb{E}_{Y}\left[	
			\norm{\mathbb{E}^N\left[\bar{\mathfrak{u}}^N_s\right] \mathbb{E}^N_{\observ}\left[\bar{\mathfrak{u}}^N_s\right]
				-
				\mathbb{E}_Y\left[\bar{u}_s\right] \mathbb{E}_Y\left[\observ\left(\bar{u}_s\right)\right]'
			}{\HHilbert^{d_y} }^2	
			\right].
		\end{align*}
		
		Again we use the conditional independence of  $\left(\bar{u}^i\right)_{i=1,\cdots,N}$ to deduce
		\begin{align*}
			&\mathbb{E}\left[
			\norm{
				\frac{1}{N}\sum_{i=1}^{N}\bar{u}^i_s \observ\left(\bar{u}^i_s\right)'
				-
				\mathbb{E}_{Y}\left[\bar{u}_s\observ\left(\bar{u}_s\right)'\right]
			}{\HHilbert^{d_y} }^2	
			\right]
			=
			\frac{\mathbb{E}\left[
				\mathbb{E}_{Y}
				\left[
				\norm{\bar{u}_s\observ\left(\bar{u}_s\right)'
					-
					\mathbb{E}_{Y}\left[\bar{u}_s\observ\left(\bar{u}_s\right)'\right]
				}{\HHilbert^{d_y}}^2
				\right]\right]}{N}
			\\&\leq
			\frac{\mathbb{E}\left[
				\mathbb{E}_{Y}
				\left[
				\norm{\bar{u}_s\observ\left(\bar{u}_s\right)'
				}{\HHilbert^{d_y}}^2
				\right]\right]}{N}
			\leq
			\frac{\norm{\observ}{\infty}\mathbb{E}\left[
				\norm{\bar{u}_s
				}{\HHilbert}^2
				\right]}{N}
		\end{align*}
		
		By using the bound \eqref{bound 2nd absolute moment} with $k,l\geq \norm{H}{\infty}$  that was derived in the proof of Theorem \ref{well posedness mf EnKBF} for the second absolute moments of $\bar{u}$, we can show that there exists a constant $C_1\left(T,\norm{\observ}{\infty}\right)>0$ such that
		\begin{align}\label{proof convergence order in N - inequality 2}
			&
			\int_{0}^{T}\mathbb{E}\left[
			\norm{
				\frac{1}{N}\sum_{i=1}^{N}\bar{u}^i_s \observ\left(\bar{u}^i_s\right)'
				-
				\mathbb{E}_{Y}\left[\bar{u}_s\observ\left(\bar{u}_s\right)'\right]
			}{\HHilbert^{d_y} }^2	
			\right]\mathrm{d}s
			\leq
			\frac{C_1\left(T,\norm{\observ}{\infty}\right)}{N}
		\end{align}
		
		Finally we note that similar to \eqref{proof convergence order in N - inequality 1} one can easily deduce 
		\begin{align*}
			&\mathbb{E}_{Y}\left[	
			\norm{\mathbb{E}^N\left[\bar{\mathfrak{u}}^N_s\right] \mathbb{E}^N_{\observ}\left[\bar{\mathfrak{u}}^N_s\right]
				-
				\mathbb{E}_Y\left[\bar{u}_s\right] \mathbb{E}_Y\left[\observ\left(\bar{u}_s\right)\right]'
			}{\HHilbert^{d_y} }^2	
			\right]
			\\&\leq
			2\mathbb{E}_{Y}\left[	
			\norm{\mathbb{E}^N\left[\bar{\mathfrak{u}}^N_s\right] 
				-
				\mathbb{E}_Y\left[\bar{u}_s\right] 
			}{\HHilbert^{d_y} }^2	
			\left|\mathbb{E}^N_{\observ}\left[\bar{\mathfrak{u}}^N_s\right]\right|^2
			\right]
			+
			2
			\norm{ \mathbb{E}_Y\left[\bar{u}_s\right] 
			}{\HHilbert^{d_y} }^2	
			\mathbb{E}_{Y}\left[	
			\left|\mathbb{E}^N_{\observ}\left[\bar{\mathfrak{u}}^N_s\right]
			-
			\mathbb{E}_Y\left[\observ\left(\bar{u}_s\right)\right]\right|^2
			\right]
			\\&\leq
			\frac{\norm{\observ}{\infty}^2
				\boundB e^{\oneSidedLip s}
			}{N}
			+
			\norm{ \mathbb{E}_Y\left[\bar{u}_s\right] 
			}{\HHilbert^{d_y} }^2	
			\frac{\mathrm{Lip}\left(\observ\right)^2\boundB e^{\oneSidedLip s}}{N}.
		\end{align*}
		Finally we can thus derive by the boundedness \eqref{bound 2nd absolute moment} of the second absolute moment of $\bar{u}$, that there exists a constant $C_2\left(T,\norm{\observ}{\infty}\right)>0$ such that
		\begin{align}\label{proof convergence order in N - inequality 3}
			\int_{0}^{T}\mathbb{E}\left[	
			\norm{\mathbb{E}^N\left[\bar{\mathfrak{u}}^N_s\right] \mathbb{E}^N_{\observ}\left[\bar{\mathfrak{u}}^N_s\right]
				-
				\mathbb{E}_Y\left[\bar{u}_s\right] \mathbb{E}_Y\left[\observ\left(\bar{u}_s\right)\right]'
			}{\HHilbert^{d_y} }^2	
			\right]\mathrm{d}s
			\leq
			\frac{C_2\left(T,\norm{\observ}{\infty}\right)}{N}
		\end{align}
		
		By the definition of $\mathrm{LLN}^{N}_{\observ}$ (in Definition \ref{Definition LLN}), combining the inequalities \eqref{proof convergence order in N - inequality 1},\eqref{proof convergence order in N - inequality 2},\eqref{proof convergence order in N - inequality 3} with \eqref{optimal rates - almost} concludes our proof.
	\end{proof}
	
	\begin{Rmk}
		Since the constant $\kappa(T,p,\nu)$ blows up for $\nu\to 1$ or $\nu\to 1/p$, we  can not simply take the limit $p\to 1$ in \eqref{optimal rates - almost}. As $p<1$, the $p$-th power is concave and thus we can not deduce that \eqref{optimal rates - almost} would also hold if the power on the left-hand-side of the inequality would be written outside! We thus say that \eqref{optimal rates - almost} gives almost optimal rates in the sense that this inequality holds for all $p<1$, but not for the optimal value $p=1$, where the the expectations on both sides are indeed of the same nature. In a similar sense the stronger inequality \eqref{optimal rates before Hoelder} shows almost optimal rates, in the sense that it holds for all $\nu>1$ but not for $\nu=1$, the case where the moments on both sides are of the same order.
	\end{Rmk}
	
	Besides blowing up when $p\to 1$, the constant $\kappa(T,p)$ also grows hyperexponentially in $T$! While Theorem \ref{Thm optimal convergence} thus provides us with (almost) optimal convergence rates, the constants involved are far too large to give useful a priori error estimates even on moderate time intervals.\newline
	
	\begin{Rmk}[Literature]\label{Literature - mean-field linear Gaussian}
		In the finite dimensional and linear Gaussian setting this problem has been tackled by a large number of papers. In this setting a first propagation of chaos result was achieved by \cite{DelMoralTugaut}, even showing uniform in time convergence for stable signals. This result has by now been followed up by several works \cite{Riccati diff - stability},\cite{Riccati diff - 1d},\cite{Riccati diff - Perturbation} treating unstable signals by making use of the Riccati equation that appears in this setting. An alternative extension of the EnKBF to nonlinear signals using a Taylor-inspired linearization around the mean, similar to the extended Kalman--Bucy filter,  was considered in  \cite{DelMoralKurtzmannTugaut}. The linearization there also allowed for the use of a decoupled Riccati equation. While an extension of these uniform in time results to our nonlinear setting is certainly highly desirable, we do not investigate it in this paper.\newline
	\end{Rmk}

	\section*{Acknowledgements and funding}
	
		Sebastian Ertel is supported by
		Deutsche Forschungsgemeinschaft through 
		\emph{IRTG 2544 - Stochastic Analysis in Interaction}.

	\clearpage
	\begin{appendices}

		\section{The Kushner-Stratonovich equation and the law of total variance}\label{section Kushner Stratonovich}\label{section variance inequality}
		
		In this section we recall the Kushner--Stratonovich equation (KSE), which describes the evloution of the posterior distribution \eqref{posterior}, and the law of total variance.\newline
		
		For the sake of brevity let us denote the (Bochner-)integral of a testfunction $\phi:\HHilbert\to\R$ w.r.t. the posterior $\eta_t,~t\geq 0$ by
		\begin{align*}
			\eta_t\left[\phi\right]:=\int_{\HHilbert} \phi(v)~\eta_t(\mathrm{d}v).
		\end{align*}
		
		With this notation in mind one can show that the posterior $\eta$ satisfies the Kushner--Stratonovich equation (KSE)
		\begin{align}\label{weak KSE}
			\mathrm{d}\eta_t[\phi]
			=
			\eta_t[\mathcal{L}\phi]\mathrm{d}t
			+
			\left(\eta_t[\phi \observ]-\eta_t[\phi]\eta_t[\observ]\right) R^{-1}_t\left(\mathrm{d}Y_t-\eta_t[\observ]\mathrm{d}t\right),
		\end{align}
		where $\mathcal{L}$ is the generator of $u$ defined in \eqref{signal generator} and $\phi$ is an arbitrary Itô testfunction (see Definition \ref{Ito function}). 
		
		\begin{Rmk}
			In finite dimensional settings this is the weak form of the KSE, also referred to as the Fujisaki-–Kallianpur-–Kunita equation. The strong form of the KSE is a nonlinear and nonlocal Fokker--Planck equation describing the evolution of the posterior density. However in our infinite dimensional setting the density (w.r.t. the Lebesgue-measure) does not exist and thus one has to work with the weak formulation \eqref{weak KSE}. The derivation of such equations for infinite dimensional filtering problems is a classical subject in stochastic analysis, studied in e.g. \cite{AhmedFuhrman},\cite{BoulangerSchiltz}, \cite{HobbsSritharan}, \cite{Yu}. However, due to the observations being finite dimensional in our case, the derivation of the KSE can also be done with classical approaches for the finite dimensional setting such as the Innovation Process Approach \cite{BainCrisan}.
		\end{Rmk}

		A key tool in our analysis of the EnKBF and its mean field limit were upper variance bounds. These upper bounds only depended on the coefficients of the signal and in particular were independent of the observation function $\observ$ and the actual observation data $Y$. This is a property that is shared by the posterior distribution \eqref{posterior}, as the projection properties of the conditional expectation imply, for $\mathbb{Cov}_{\eta_t}\left[\mathrm{id}_{\HHilbert}\right]$ denoting the covariance operator of the posterior, the inequality
		\begin{align}\label{law of total variance - direct inequality}
			\mathbb{E}\left[~\eta_t(\norm{\cdot}{\HHilbert}^2)-
			\norm{\eta_t(\mathrm{id}_{\HHilbert})}{\HHilbert}^2
			~\right]
			=
			\mathbb{E}\left[~\mathrm{tr}_{\HHilbert}~\mathbb{Cov}_{\eta_t}\left[\mathrm{id}_{\HHilbert}\right]~\right]
			\leq
			\mathrm{tr}_{\HHilbert}~\mathbb{Cov}\left[u_t\right].
		\end{align}

		Now we note that by the covariance dynamics \eqref{equation signal covariance} and Parseval\footnote{One could have also directly looked at the dynamics of $\mathbb{E}\left[\norm{u_t-m_t}{\HHilbert}^2\right]$ and proved this via the well known Itô formula for the norm.} we obtain
		\begin{align*}
			\partial_t \mathrm{tr}_{\HHilbert}~\mathbb{Cov}\left[u_t\right]
			&=
			2 \sum_{k\in\N}\mathbb{E}\left[\left\langle \nu_k,u_t-m_t\right\rangle_{\HHilbert}
			\prescript{}{\VHilbert'}{\left\langle \Adrift(u_t)-\Adrift(m_t),\nu_k
				\right\rangle_{\VHilbert}}\right]
			\\&\phantom{=}+
			\sum_{k\in\N}
			\left\langle \nu_k, \mathbb{E}
			\left[\Bvol(u_t)\sqrt{\Qcor}
			\left(\Bvol(u_t)\sqrt{\Qcor}\right)'
			\right] \nu_k
			\right\rangle_{\HHilbert}
			\\&=
			\mathbb{E}\left[2
			\prescript{}{\VHilbert'}{\left\langle \Adrift(u_t)-\Adrift(m_t),u_t-m_t
				\right\rangle_{\VHilbert}}\right]
			+
			\mathbb{E}\left[
			\mathrm{tr}_{\HHilbert}\left[\Bvol(u_t)\sqrt{\Qcor}
			\left(\Bvol(u_t)\sqrt{\Qcor}\right)'\right]
			\right]
		\end{align*}

		To estimate the first term we use the one-sided Lipschitz condition \eqref{one sided Lipschitz} and for the second term we use Assumption \ref{bounded signal diffusion}. This gives us
		\begin{align*}
			\partial_t \mathrm{tr}_{\HHilbert}~\mathbb{Cov}\left[u_t\right]
			\leq
			\oneSidedLip~
			\mathrm{tr}_{\HHilbert}~\mathbb{Cov}\left[u_t\right]
			+
			\boundB.
		\end{align*}
		
		Which, by Grönwall, implies $
		\mathrm{tr}_{\HHilbert}~\mathbb{Cov}\left[u_t\right] \leq \boundB e^{\oneSidedLip t}$. Together with the variance bound for the posterior this gives us
		\begin{align}\label{variance bound for posterior}
			\mathbb{E}\left[~\eta_t(\norm{\cdot}{\HHilbert}^2)-
			\norm{\eta_t(\mathrm{id}_{\HHilbert})}{\HHilbert}^2
			~\right]
			=
			\mathbb{E}\left[~\mathrm{tr}_{\HHilbert}~\mathbb{Cov}_{\eta_t}\left[\mathrm{id}_{\HHilbert}\right]~\right]
			\leq
			\boundB e^{\oneSidedLip t},
		\end{align}
		which is the same  bound that was shown for the empirical variance of the EnKBF in \eqref{variance inequality particle system - in expectation} and for the (conditional) variance of the mean field EnKBF \eqref{variance bound tilde u}. Note in particular that this bound is independent of, and thus robust in, both the observation data $Y$ and the observation function $\observ$.	In probability theory the identity 
		\begin{align*}
			\mathbb{E}\left[\mathbb{Var}\left[X|Y\right]\right]
			+
			\mathbb{Var}\left[\mathbb{E}\left[X|Y\right]\right]
			=
			\mathbb{Var}\left[X\right]
		\end{align*}
		for any two random variables $X,Y$, which implies inequality \eqref{law of total variance - direct inequality}, is often referred to as the law of total variance. As we have seen, inequality \eqref{variance bound for posterior} is a direct consequence of this identity and thus the EnKBF (and its mean field limit)  also satisfies a form of law of total variance.\newline

		This is not the only connection of the EnKBF to the filtering problem, as in the linear Gaussian setting (see Definition \ref{linear Gaussian setting} below) the law of its mean field limit \eqref{mean field EnKBF - introduction} is a solution to the KSE and thus coincides with the posterior \eqref{posterior} (assuming that solutions to the KSE are unique).
		
		\begin{Def}\label{linear Gaussian setting}[Linear Gaussian Setting]
			When we speak of the linear Gaussian setting we mean that
			\begin{itemize}
				\item $\Adrift:\VHilbert\to\VHilbert'$ is linear
				
				\item $\Bvol\in\LinSpace{\UHilbert}{\HHilbert}$ is a constant linear operator, i.e. it is independent of the state $u$.
				
				\item $H\in\LinSpace{\HHilbert}{\R^{d_y}}$ is a linear operator
				
				\item the initial condition (of the signal and posterior) $u_0$ is a Gaussian on $\HHilbert$, i.e.
				\begin{align*}
					u_0\sim\mathcal{N}\left(m_0,P_0\right)
					~\text{with}~
					m_0\in\HHilbert
					~\text{and}~
					P_0\in\LinSpace{\HHilbert}{\HHilbert}~\text{is symmetric positive semidefinite}.
				\end{align*}
			\end{itemize}
		\end{Def}
		
		In this linear Gaussian setting one can show (see e.g. \cite{Yu}) that the solution of the KSE is a Gaussian $\mathcal{N}\left(m_t,P_t\right)$, with mean $m$ and covariance $P$ given by the famous Kalman--Bucy equation
		\begin{subequations}\label{Kalman Bucy}
			\begin{equation}\label{linear mean evolution}
				\mathrm{d}m_t
				=
				\Adrift m_t\mathrm{d}t+ P_t \observ' R^{-1}_t\left(\mathrm{d}Y_t-\observ m_t\mathrm{d}t\right)
			\end{equation}
			\begin{equation}\label{Riccati equation}
				\frac{\mathrm{d}P_t}{\mathrm{d}t}
				=
				\Adrift P_t + P_t \Adrift'
				-
				P_t \observ' R^{-1}_t \observ P_t
				+
				\left(\Bvol\sqrt{\Qcor}\right)
				\left(\Bvol\sqrt{\Qcor}\right)'
				.
			\end{equation}
		\end{subequations}
		
		In our infinite dimensional variational setting these equations \eqref{Kalman Bucy} were first studied by Bensoussan \cite[Théorème 3.1]{Bensoussan - filtrage optimal}, proving existence and uniqueness of solutions. Now note that one can easily show that in the linear Gaussian setting of Definition \ref{linear Gaussian setting}, the mean field EnKBF \eqref{mean field EnKBF - introduction}, which we write here again as
		\begin{align}\label{linear EnKBF}
			\mathrm{d}\bar{u}_t=\Adrift \bar{u}_t \mathrm{d}t+\Bvol\mathrm{d}\bar{W}_t
			+
			\bar{P}_t \observ' R^{-1}_t\left(\mathrm{d}Y_t-H\frac{\bar{u}_t+\bar{m}_t}{2}
			\mathrm{d}t
			\right),
		\end{align}
		defines an (infinite dimensional) Ornstein--Uhlenbeck process. As one can easily verify with Itô's formula, its (conditional) mean $\bar{m}$ and covariance $\bar{P}$ satisfy the Kalman--Bucy equations \eqref{Kalman Bucy}. Thus by the uniqueness of \eqref{Kalman Bucy}, the (conditional) law $\bar{\eta}_t$ of $\bar{u}_t$ has to coincide with the posterior \eqref{posterior}!
		
		\section{The Feedback Particle Filter}\label{section FPF}
		
		In the linear Gaussian setting (Definition \ref{linear Gaussian setting}) the mean field EnKBF \eqref{linear EnKBF} describes a diffusion process with the remarkable property, that its (conditional) time-marginal laws are given by the desired posterior distribution. In the general setting this attribute is achieved by the Feedback Particle Filter (FPF). It is given by
		\begin{align}\label{FPF}
			\begin{split}
				\mathrm{d}\hat{u}_t
				&=\Adrift(\hat{u}_t)\mathrm{d}t+\Bvol(\hat{u}_t)\mathrm{d}\bar{W}_t
				\\&\phantom{=}+
				\Kalman(\hat{u}_t,\hat{\eta}_t) \left(\mathrm{d}Y_t-\frac{\observ(\hat{u}_t)+\mathbb{E}_Y\left[\observ(\hat{u}_t)\right]}{2}
				\mathrm{d}t
				\right)
				+\frac{1}{2}\xi(\hat{u}_t,\hat{\eta}_t)\mathrm{d}t,
			\end{split}
		\end{align}
		where $\hat{\eta}_t$ denotes the conditional distribution of $\hat{u}_t$. The so called gain term $\Kalman(\cdot,\hat{\eta}_t):\HHilbert\to \HHilbert\times\R^{d_y}$ is (not uniquely) determined by the weak differential equation
		\begin{align}\label{Kalman gain equation}
			\hat{\eta}_t\left[\left\langle\Frechet{\HHilbert}{\phi}~,~\Kalman(\cdot,\hat{\eta}_t)\right\rangle_{\HHilbert}\right]
			=
			\hat{\eta}_t\left[\phi\left(\observ-\hat{\eta}_t\left[\observ\right]\right)'\right]R^{-1}_t
			~\text{for all Itô testfunctions}~\phi.
		\end{align}
		Just as in section \ref{section problem setting} we  denote by $\Frechet{\HHilbert}{\phi}$ and $\Hesse{\HHilbert}{\phi}$ the first and second order Fréchét  derivatives on the Hilbert space $\HHilbert$. 
		The correctional drift term $\xi(\cdot,\hat{\eta}_t):\HHilbert\to \HHilbert$ is given by
		\begin{align*}
			\xi(\hat{u}_t,\hat{\eta}_t)
			&:=
			\left(
			\left\langle \Kalman\left(\hat{u}_t,\hat{\eta}_t\right),\nabla\right\rangle_{\mathscr{H}} R_t
			\Kalman\left(\hat{u}_t,\hat{\eta}_t\right)^{'}
			\right)^{'}
			\\&:=
			\sum_{j\in\N}\sum_{k\in\N}
			\left\langle\nu_j,\partial_{\nu_k}\Kalman(\hat{u}_t,\hat{\eta}_t)\right\rangle_{\HHilbert}R_t
			\left\langle\nu_k,\Kalman(\hat{u}_t,\hat{\eta}_t)\right\rangle_{\HHilbert}^{\mathrm{T}}
			\nu_j
			,
		\end{align*}
		where $\partial_{\nu_k}$ denotes the derivative in direction of the basis vector $\partial_{\nu_k}$.\newline
		
		In the finite dimensional setting the FPF was first derived in \cite{YangMehtaMeyn1} with an optimal control approach, independently and prior to this work a similar mean-field optimal filter for smoothed noise and finite dimensional signals has been found in \cite{CrisanXiong}. In \cite{PathirajaStannatReich} various finite dimensional consistent mean field filters, among them the original FPF, have been derived by matching the strong Fokker--Planck equation of a diffusion process to the strong form of the KSE \eqref{weak KSE}. Building on this work we now extend the FPF to infinite dimensions by showing that it describes the optimal filter, in the sense that the (conditional) law of \eqref{FPF} propagates in time exactly according to the KSE \eqref{weak KSE}. However, since we are working in the infinite setting, we do so by matching the weak Fokker--Planck equation to the weak KSE \eqref{weak KSE}.\newline
		
		\begin{Rmk}
			The well posedness of the FPF \eqref{FPF} is an open problem, even in the much simpler finite dimensional case, and is thus just assumed in the following.
		\end{Rmk}

		\begin{Lemma}
			Denote by $\left(\hat{\eta}_t\right)_{t\geq 0}$ the law of the FPF $\left(\hat{u}_t\right)_{t\geq 0}$, given by \eqref{FPF}. We assume that for all times $t\geq 0$ both $\Kalman(\cdot,\hat{\eta}_t)$ and $\xi(\cdot,\hat{\eta}_t)$ are well defined functions from $\HHilbert$ into $\HHilbert$, that are integrable with respect to $\hat{\eta}_t$.\\
			Then $\hat{\eta}$ satisfies the KSE \eqref{weak KSE} for all Itô testfunctions (according to Definition \ref{Ito function}) $\phi:\HHilbert\to\R$, with the following properties
			\begin{itemize}			
				\item $\phi$ is integrable with respect to $\hat{\eta}_t$ for all $t\geq 0$.
				
				\item For all $v\in H$ the Hessian $\Hesse{\HHilbert}{\phi}(v)$ is a self adjoint operator on $\HHilbert$.
				
				\item The map $\hat{\Phi}:\HHilbert\to\R^{d_y}$, defined by $\hat{\Phi}(v):=\left\langle \Frechet{\HHilbert}{\phi}(v),\Kalman(v,\hat{\eta}_t)\right\rangle_{\HHilbert}R_t$, is an Itô function (componentwise) that is also integrable with respect to $\eta_t$ for all $t\geq 0$.

			\end{itemize}
		\end{Lemma}
		\begin{proof}
			Let $\phi:\HHilbert\to\R$ be arbitrary, satisfying the properties specified above. We note by Itô's formula, that the Kolmogorov forward equation describing the evolution of $\hat{\eta}$ is given by
			\begin{align}\label{Fokker Planck of FPF}
				\begin{split}
					\mathrm{d}\hat{\eta}_t[\phi]
					&=
					\hat{\eta}_t\left[\mathcal{L}\phi\right]\mathrm{d}t
					+
					\hat{\eta}_t\left[\left\langle\Frechet{\HHilbert}{\phi},\Kalman(\cdot,\hat{\eta}_t)\right\rangle_{\HHilbert}\right]
					\left(\mathrm{d}Y_t-\hat{\eta}_t[\observ]\mathrm{d}t\right)
					\\&\phantom{=}+
					\frac{
						\hat{\eta}_t\left[\left\langle\Frechet{\HHilbert}{\phi},\xi(\cdot,\hat{\eta}_t)\right\rangle_{\HHilbert}\right]
						+
						\hat{\eta}_t\left[
						\mathrm{tr}_{\HHilbert}\left[\left(\Hesse{\HHilbert}{\phi}\right) \Kalman(\cdot,\hat{\eta}_t) R_t \Kalman(\cdot,\hat{\eta}_t)' \right]
						\right]
					}{2}\mathrm{d}t
					\\&\phantom{=}-
					\frac{\hat{\eta}_t\left[\left\langle\Frechet{\HHilbert}{\phi},\Kalman(\cdot,\hat{\eta}_t)(\observ-\hat{\eta}_t(\observ))\right\rangle_{\HHilbert}\right]}{2}\mathrm{d}t
				\end{split}
			\end{align}
			
			Due to \eqref{Kalman gain equation}, the first line of \eqref{Fokker Planck of FPF} is exactly the KSE and thus, to show consistency, we only have to prove that the second line is zero.  To this end we note that by Parseval we have
			\begin{align*}
				&\left\langle\Frechet{\HHilbert}{\phi},\xi(\cdot,\hat{\eta}_t)\right\rangle_{\HHilbert}
				+
				\mathrm{tr}_{\HHilbert}\left[\left(\Hesse{\HHilbert}{\phi}\right) \Kalman(\cdot,\hat{\eta}_t) R_t \Kalman(\cdot,\hat{\eta}_t)' \right]
				\\&=
				\sum_{j\in\N}
				\left\langle\nu_j,\Frechet{\HHilbert}{\phi}\right\rangle_{\HHilbert}
				\left\langle\nu_j,\xi(\cdot,\hat{\eta}_t)\right\rangle_{\HHilbert}
				+
				\sum_{k\in\N}
				\left\langle\nu_k,\left(\Hesse{\HHilbert}{\phi}\right) \Kalman(\cdot,\hat{\eta}_t) R \Kalman(\cdot,\hat{\eta}_t)'\nu_k\right\rangle_{\HHilbert}
				\\&=
				\sum_{j\in\N}
				\sum_{k\in\N}
				\left\langle\nu_j,\Frechet{\HHilbert}{\phi}\right\rangle_{\HHilbert}
				\left\langle\nu_j,\partial_{\nu_k}\Kalman(\hat{u}_t,\hat{\eta}_t)\right\rangle_{\HHilbert} R_t
				\left\langle\nu_k,\Kalman(\hat{u}_t,\hat{\eta}_t)\right\rangle_{\HHilbert}^{\mathrm{T}}
				\\&\phantom{=}+
				\sum_{k\in\N}
				\left\langle\nu_k,\left(\Hesse{\HHilbert}{\phi}\right) \Kalman(\cdot,\hat{\eta}_t)\right\rangle_{\HHilbert} R_t
				\left\langle \nu_k,\Kalman(\cdot,\hat{\eta}_t)\right\rangle_{\HHilbert}^{\mathrm{T}}
			\end{align*}
			
			Next we now that since $\Hesse{\HHilbert}{\phi}$ is self adjoint and by Parseval we have
			\begin{align*}
				\left\langle\nu_k,\Hesse{\HHilbert}{\phi} \Kalman(\cdot,\hat{\eta}_t)\right\rangle_{\HHilbert}
				&=
				\left\langle\Hesse{\HHilbert}{\phi}\nu_k, \Kalman(\cdot,\hat{\eta}_t)\right\rangle_{\HHilbert}
				=
				\left\langle\partial_{\nu_k}\Frechet{\HHilbert}{\phi}, \Kalman(\cdot,\hat{\eta}_t)\right\rangle_{\HHilbert}
				\\&=
				\sum_{j\in\N}
				\left\langle\nu_j,\partial_{\nu_k}\Frechet{\HHilbert}{\phi}\right\rangle_{\HHilbert}
				\left\langle\nu_j, \Kalman(\cdot,\hat{\eta}_t)\right\rangle_{\HHilbert}.
			\end{align*}
			
			This gives us by again using Parseval
			\begin{align*}
				&\left\langle\Frechet{\HHilbert}{\phi},\xi(\cdot,\hat{\eta}_t)\right\rangle_{\HHilbert}
				+
				\mathrm{tr}_{\HHilbert}\left[\left(\Hesse{\HHilbert}{\phi}\right) \Kalman(\cdot,\hat{\eta}_t) R_t \Kalman(\cdot,\hat{\eta}_t)' \right]
				\\&=
				\sum_{k\in\N}
				\sum_{j\in\N}
				\left\langle\nu_j,\Frechet{\HHilbert}{\phi}\right\rangle_{\HHilbert}
				\left\langle\nu_j,\partial_{\nu_k}\Kalman(\hat{u}_t,\hat{\eta}_t)\right\rangle_{\HHilbert}R_t
				\left\langle\nu_k,\Kalman(\hat{u}_t,\hat{\eta}_t)\right\rangle_{\HHilbert}^{\mathrm{T}}
				\\&\phantom{=}+
				\sum_{k\in\N}
				\sum_{j\in\N}
				\left\langle\nu_j,\partial_{\nu_k}\Frechet{\HHilbert}{\phi}\right\rangle_{\HHilbert}
				\left\langle\nu_j, \Kalman(\cdot,\hat{\eta}_t)\right\rangle_{\HHilbert} R_t
				\left\langle \nu_k,\Kalman(\cdot,\hat{\eta}_t)\right\rangle_{\HHilbert}^{\mathrm{T}}
				\\&=\sum_{k\in\N}
				\left(
				\left\langle\Frechet{\HHilbert}{\phi},\partial_{\nu_k}\Kalman(\cdot,\hat{\eta}_t)\right\rangle_{\HHilbert}
				+
				\left\langle\partial_{\nu_k}\Frechet{\HHilbert}{\phi},\Kalman(\cdot,\hat{\eta}_t)\right\rangle_{\HHilbert}
				\right)
				R_t
				\left\langle \nu_k,\Kalman(\cdot,\hat{\eta}_t)\right\rangle_{\HHilbert}^{\mathrm{T}}
				.
			\end{align*}
			
			Using the product formula in Hilbert spaces\footnote{More precisely the formula for the directional derivative of the scalar product of two differentiable functions.}, we thus derive
			\begin{align*}
				&\left\langle\Frechet{\HHilbert}{\phi},\xi(\cdot,\hat{\eta}_t)\right\rangle_{\HHilbert}
				+
				\mathrm{tr}_{\HHilbert}\left[\left(\Hesse{\HHilbert}{\phi}\right) \Kalman(\cdot,\hat{\eta}_t) R_t \Kalman(\cdot,\hat{\eta}_t)' \right]
				\\&=
				\sum_{k\in\N}
				\partial_{\nu_k}\left\langle \Frechet{\HHilbert}{\phi},\Kalman(\cdot,\hat{\eta}_t)\right\rangle_{\HHilbert} R_t
				\left\langle \nu_k,\Kalman(\cdot,\hat{\eta}_t)\right\rangle_{\HHilbert}^{\mathrm{T}}.
			\end{align*}
			
			The map $\hat{\Phi}$, defined in the statement of the Lemma, then allows us to again use Parseval to derive
			\begin{align*}
				&\left\langle\Frechet{\HHilbert}{\phi},\xi(\cdot,\hat{\eta}_t)\right\rangle_{\HHilbert}
				+
				\mathrm{tr}_{\HHilbert}\left[\left(\Hesse{\HHilbert}{\phi}\right) \Kalman(\cdot,\hat{\eta}_t) R_t \Kalman(\cdot,\hat{\eta}_t)' \right]
				\\&=
				\sum_{k\in\N}
				\left(\partial_{\nu_k}\hat{\Phi}\right)
				\left\langle \nu_k,\Kalman(\cdot,\hat{\eta}_t)\right\rangle_{\HHilbert}^{\mathrm{T}}
				=
				\sum_{k\in\N}
				\left\langle\nu_k, \Frechet{\HHilbert}{\hat{\Phi}}\right\rangle_{\HHilbert}
				\left\langle \nu_k,\Kalman(\cdot,\hat{\eta}_t)\right\rangle_{\HHilbert}^{\mathrm{T}}
				\\&=
				\sum_{i=1}^{d_y}\sum_{k\in\N}\sum_{i=1}^{d_y}
				\left\langle\nu_k, \Frechet{\HHilbert}{\hat{\Phi}}\right\rangle_{\HHilbert}
				\delta_i\delta_i^\mathrm{T}
				\left\langle \nu_k,\Kalman(\cdot,\hat{\eta}_t)\right\rangle_{\HHilbert}^{\mathrm{T}}
				=
				\sum_{i=1}^{d_y}
				\left\langle\Frechet{\HHilbert}{\hat{\Phi}}\delta_i,\Kalman(\cdot,\hat{\eta}_t) \right\rangle_{\HHilbert}\delta_i,
			\end{align*}
			where $\delta_i,~i=1,\cdots,d_y$ denotes the canonical basis of $\R^{d_y}$.	Thus, by the assumed regularity of $\hat{\Phi}$, we derive by using \eqref{Kalman gain equation}
			\begin{align*}
				&\hat{\eta}_t\left[\left\langle\Frechet{\HHilbert}{\phi},\xi(\cdot,\hat{\eta}_t)\right\rangle_{\HHilbert}\right]
				+
				\hat{\eta}_t\left[\mathrm{tr}_{\HHilbert}\left[\left(\Hesse{\HHilbert}{\phi}\right) \Kalman(\cdot,\hat{\eta}_t) R_t\Kalman(\cdot,\hat{\eta}_t)' \right]\right]
				\\&=\sum_{i=1}^{d_y}
				\hat{\eta}_t\left[\left\langle \Frechet{\HHilbert}{\hat{\Phi}}\delta_i,\Kalman(\cdot,\hat{\eta}_t)\right\rangle_{\HHilbert}\right]\delta_i
				=
				\sum_{i=1}^{d_y}
				\hat{\eta}_t\left[\hat{\Phi}\delta_i\left(\observ-\hat{\eta}[\observ]\right)'\right]R^{-1}_t\delta_i
				\\&=
				\sum_{i=1}^{d_y}
				\hat{\eta}_t\left[
				\left\langle \Frechet{\HHilbert}{\phi},\Kalman(\cdot,\hat{\eta}_t)\right\rangle_{\HHilbert}R_t
				\delta_i\left(\observ-\hat{\eta}[\observ]\right)'\right] R^{-1}_t\delta_i
				\\&=
				\sum_{i=1}^{d_y}
				\hat{\eta}_t\left[
				\left\langle \Frechet{\HHilbert}{\phi},\Kalman(\cdot,\hat{\eta}_t)\right\rangle_{\HHilbert}R_t
				\delta_i\delta_i^{\mathrm{T}}R^{-1}_t\left(\observ-\hat{\eta}[\observ]\right)'\right]
				\\&=
				\hat{\eta}_t\left[
				\left\langle \Frechet{\HHilbert}{\phi},\Kalman(\cdot,\hat{\eta}_t)\left(\observ-\hat{\eta}[\observ]\right)\right\rangle_{\HHilbert}\right]
				,
			\end{align*}
			which in turn lets us  conclude that \eqref{Fokker Planck of FPF} coincides with the KSE \eqref{weak KSE} and thus the FPF is indeed consistent.
		\end{proof}

		The FPF is a true generalization of the EnKBF to general filtering problems and it even provides a connection between the EnKBF and the true posterior even in inconsistent setting as the following Lemma shows.
		
		\begin{Lemma}
			Let again $\left(\hat{\eta}\right)_{t\geq 0}$ be the (conditional) marginal laws to the FPF \eqref{FPF}. Assuming integrability of $\Kalman(\cdot,\hat{\eta}_t)$, then it holds that
			\begin{align}\label{constant gain approximation}
				\mathbb{E}_Y\left[\Kalman(\hat{u}_t,\hat{\eta}_t)\right]
				=
				\hat{\eta}_t\left( \Kalman(\cdot,\hat{\eta}_t) \right)
				=
				\mathbb{C}_{\hat{\eta}_t}\left[\mathrm{id}_{\HHilbert}, \observ\right]R^{-1}_t
				=
				\mathbb{C}_Y\left[\hat{u}_t, \observ(\hat{u}_t)\right]R^{-1}_t.
			\end{align}
			{ If $\observ$ is linear and $\hat{\eta}_t$ is Gaussian, one can even choose the gain term $K$ such that $\Kalman(\cdot,\hat{\eta}_t)=\mathbb{C}_{\hat{\eta}_t}\left[\mathrm{id}_{\HHilbert}, \observ\right] R^{-1}_t$. In the linear Gaussian setting the EnKBF is thus just a special case of the FPF.}
		\end{Lemma}
		\begin{proof}
			For any $i\in\N$ we set $\phi_i(v):=\left\langle\nu_i,v\right\rangle_{\HHilbert}$ as a testfunction in the gain equation \eqref{Kalman gain equation}, then we have
			\begin{align*}
				&\left\langle\nu_i,\hat{\eta}_t\left(\Kalman(\cdot,\hat{\eta}_t)\right)\right\rangle_{\HHilbert}
				=
				\hat{\eta}_t\left(\left\langle\Frechet{\HHilbert}{\phi_i}~,~\Kalman(\cdot,\hat{\eta}_t)\right\rangle_{\HHilbert}\right)
				=
				\hat{\eta}_t\left(\phi_i\left(\observ-\hat{\eta}[\observ]\right)'\right)R^{-1}_t
				\\&=
				\left\langle\nu_i,\mathbb{C}_{\hat{\eta}_t}\left[\mathrm{id}_{H}, \observ\right]R^{-1}_t\right\rangle_{\HHilbert}.
			\end{align*}
			
			Since this holds for any $i\in\N$ this indeed shows the validity of \eqref{constant gain approximation}. The second claim follows from Gaussian integration by parts as in \cite{Bogachev}.
		\end{proof}.
		
		Identity \eqref{constant gain approximation} is the reason why the mean field EnKBF  is sometimes referred to as the constant gain approximation (to the FPF) \cite{TaghvaeiDeWiljesEtAl}. It may be of interest for the statistical analysis of the EnKBF, as it provides a link to the Bayesian filtering problem. However such an analysis is out of the scope of this paper and has so far only been attempted in \cite{CarrilloHoffmannStuartVaes} in finite and time discrete settings under restrictive assumptions.  	
	\end{appendices}

\end{document}